\documentclass[11pt]{article}
\usepackage[utf8]{inputenc}
\usepackage[english]{babel}
\usepackage[a4paper,margin=1in]{geometry}
\usepackage{amsmath,amssymb,amsthm,mathtools}
\usepackage{enumitem}
\usepackage{booktabs}
\usepackage{graphicx}
\usepackage{xcolor}
\usepackage{tikz}
\usetikzlibrary{arrows.meta, positioning, decorations.pathreplacing, decorations.pathmorphing, calc, fit}
\usepackage{authblk}
\usepackage{orcidlink}
\usepackage{verbatim}
\usepackage{hyperref}
\hypersetup{colorlinks=true,linkcolor=blue!50!black,citecolor=blue!50!black,urlcolor=blue!50!black}

\newtheorem{theorem}{Theorem}[section]
\newtheorem{proposition}[theorem]{Proposition}
\newtheorem{lemma}[theorem]{Lemma}
\newtheorem{corollary}[theorem]{Corollary}
\theoremstyle{definition}

\newtheorem{example}[theorem]{Example}
\theoremstyle{remark}
\newtheorem{remark}[theorem]{Remark}

\newcommand{\Intn}{\mathrm{Int}_n}

\newcommand{\HM}{\mathrm{HM}}
\newcommand{\Star}{\mathrm{Star}}

\title{\large\textbf{A gap theorem for non-trivial maximal intersecting
families and an exact weighted asymptotic}}

\author[1,2]{\textbf{El'mira Yu.\ Kalimulina}\orcidlink{0000-0001-7158-040X}\thanks{\texttt{eyk@iitp.ru}}}

\affil[1]{Institute for Information Transmission Problems
of the Russian Academy of Sciences (Kharkevich Institute), Moscow 127051, Russia}
\affil[2]{Lomonosov Moscow State University, Moscow, Russia}

\date{}

\begin{document}
\maketitle

\begin{abstract}
\emergencystretch=1.5em
Let $D_n$ be the disjointness graph on the nonempty subsets of
$[n]=\{1,\dots,n\}$, whose independent sets are exactly the
intersecting families on $[n]$. We study the weighted independent-set
polynomial $W(n)=\sum_{\mathcal F}\prod_{S\in\mathcal F}w(S)$ for the
doubly exponential weight $w(S)=2^{2^{n-|S|}}-1$. The kernel-bearing
(trivial) contribution is exact by inclusion--exclusion,
$Z_\cap(n)=\sum_{j=1}^{n}(-1)^{j+1}\binom nj\,2^{3^{n-j}}$
$\sim n\cdot2^{3^{n-1}}$, and we determine the leading suppression exponent
of the kernel-free remainder $R(n)$. We prove the exact prefactor
\[
R(n)=\Bigl(\tfrac34+o(1)\Bigr)\,n\cdot2^{\,3^{n-1}-2^{n-1}+2},
\qquad\text{whence}\qquad
\log_2\frac{Z_\cap(n)}{R(n)}=2^{n-1}-2+\log_2\tfrac43+o(1),
\]
an \emph{additive} $o(1)$ rather than only a leading-order one. The
engine is a second-level extremal theorem: among kernel-free maximal
linked systems other than the $n$ one-flip stars, the largest weight
exponent is $3^{n-1}-3\cdot2^{n-2}+6$, a fixed gap $2^{n-2}-4$ below
the maximum, with the extremisers classified exactly. None of this is
special to the weight $2^{2^{n-|S|}}-1$: for the entire family
$w_B(S)=B^{\,B^{\,n-|S|}}-1$ ($B\in\mathbb Z_{\ge2}$) the same stars
dominate, the same near-extremal families sit a gap $B^{\,n-2}-B^{2}$
below, and the prefactor is the explicit $1-B^{-B}$. The
combinatorial input is the $p$-biased extremal problem for
non-trivial intersecting families: for every $n\ge3$ and
$0<p\le\tfrac12$,
\[
M_2(n,p)\;=\;p-pq^{\,n-1}+qp^{\,n-1},\qquad q=1-p.
\]
This first level is essentially known: the extremal family is the
Wheel coterie of Peleg and Wool, and at $p=1/Q$ the statement, with
its maximiser classification, is the case $r=n$ of Borg's
Hilton--Milner theorem for signed sets (2013). We include a short
self-contained Erd\H{o}s--Ko--Rado proof, uniform in real
$p\in(0,\tfrac12]$, whose layer-two rigidity is what the second-level
theorem consumes. The novelty claimed lies at the second level and in
the prefactor; there the \emph{classification} cannot be recovered
from the layer profile alone -- at $n=5$ one layer profile carries two
non-isomorphic types of second-level extremisers, so no function of
the profile can list them.
\end{abstract}



\section{Introduction}\label{sec:intro}

\subsection{The model}\label{sec:model}

Let $[n]=\{1,\dots,n\}$ and let $V_n:=2^{[n]}\setminus\{\varnothing\}$ be
the set of its nonempty subsets. The \emph{disjointness graph} $D_n$ has
vertex set $V_n$, with $S\sim T$ whenever $S\cap T=\varnothing$. A set
$\mathcal F\subseteq V_n$ is an \emph{independent set} of $D_n$ --
equivalently, no two of its members are disjoint -- if and only if
$\mathcal F$ is an \emph{intersecting family}: $A\cap B\neq\varnothing$
for all $A,B\in\mathcal F$, not necessarily distinct; in particular an
intersecting family contains no empty set. We call $\mathcal F$
\emph{non-trivial} (or \emph{kernel-free}) if in addition
$\ker\mathcal F:=\bigcap_{F\in\mathcal F}F=\varnothing$, and
\emph{trivial} (\emph{kernel-bearing}) otherwise -- all members then
share a common element $i\in[n]$, and we write
$\Star(i)=\{F\in V_n:i\in F\}$ for the (unique, inclusion-maximal)
intersecting family with $i$ in its kernel. 

Given a weight function $w:V_n\to\mathbb R_{>0}$, the associated
\emph{weighted independent-set polynomial} 
is
\begin{equation}\label{eq:W}
Z(D_n,w)\;:=\;\sum_{\mathcal F\in\mathcal I(D_n)}\;\prod_{S\in\mathcal
F}w(S),
\end{equation}
the sum running over all independent sets $\mathcal F$ of $D_n$
(equivalently, all intersecting families of nonempty subsets of $[n]$),
with the empty product assigning weight~$1$ to $\mathcal F=\varnothing$.
We write $\Intn:=\mathcal I(D_n)$ for this set of intersecting families
(including $\varnothing$, whose kernel is $[n]$ by the convention
$\ker\varnothing:=[n]$).
This is the object of study of the present paper, for two related choices
of $w$: a doubly exponential, highly non-uniform one
(\S\ref{sec:main-result}), and the classical uniform-per-element one
(\S\ref{sec:extremal-input}).

\subsection*{Main results}

The central part of the paper is a rigidity phenomenon for maximal linked
systems and its enumerative consequence. Below we summarise the main
theorems and their logical dependency. %

\begin{itemize}[leftmargin=*]
\item[\textbf{A.}] \emph{(Second-level gap; Theorem~\ref{thm:second}.)}
For $n\ge5$, every kernel-free maximal linked system other than the $n$
one-flip stars $\mathcal F_a^{*}$ has weight exponent
$\Lambda(M)\le M_n^{*}-(2^{n-2}-4)$, a fixed gap below the maximum
$M_n^{*}=3^{n-1}-2^{n-1}+2$, and the extremisers are classified exactly
($n(n-1)$ punctured-star systems, plus $10$ Ahlswede--Khachatrian balls when
$n=5$). The division of labour with profile methods is made precise in
Proposition~\ref{prop:oneprofile} and Remark~\ref{rem:secondgerbner}:
the two \emph{values} $M_n^{*}$ and $M_n^{*}-(2^{n-2}-4)$ follow from a
linear optimisation over layer profiles, while the
\emph{classification} provably does not -- at $n=5$ the second level is
carried by a \emph{single} profile realised by two non-isomorphic
family types, so no profile-based description can output the list of
extremisers.
\item[\textbf{B.}] \emph{(Exact prefactor; Theorem~\ref{thm:prefactor},
Corollary~\ref{cor:additive}.)} The second-level gap forces the exact
prefactor of the kernel-free remainder,
\[
R(n)=\Bigl(\tfrac34+o(1)\Bigr)n\cdot2^{M_n^{*}},
\qquad
\log_2\frac{Z_\cap(n)}{R(n)}=2^{n-1}-2+\log_2\tfrac43+o(1),
\]
an \emph{additive} sharpening of the leading exponent
(Theorem~\ref{thm:main}).
\item[\textbf{C.}] \emph{(Extremal input; Theorem~\ref{thm:main-mu}.)}
The gap in A rests on the exact $p$-biased maximum
$M_2(n,p)=p-pq^{\,n-1}+qp^{\,n-1}$ over non-trivial intersecting
families and the rigidity Lemma~\ref{lem:rigidity} extracted from its
proof. This first level is essentially known: the extremal family is the
Wheel coterie of the quorum-availability literature, in which $\mu_p$
is the system failure probability studied by Peleg and
Wool~\cite{PelegWool1995}, and at every rational $p=1/Q$, $Q\ge3$,
the value follows from, and the maximiser list coincides with, Borg's
Hilton--Milner theorem for signed sets~\cite{Borg2013} (restated
in~\cite{KwanSudakovVieira2018}; the $t=1$ case
of~\cite{FranklNie2026,HouHu2026}), as \S\ref{sec:known} details.
Theorem~\ref{thm:main-mu} contributes a short self-contained
Erd\H os--Ko--Rado proof, uniform in real $p\in(0,\tfrac12]$, and --
crucially for A -- the layer-two rigidity that this proof yields.
\end{itemize}

\noindent The dependency is
$\text{C}\Rightarrow\text{A}\Rightarrow\text{B}$: the extremal theorem
supplies $M_n^{*}$ and the rigidity lemma; rigidity gives the
second-level gap; the gap gives the exact prefactor. Thus $M_2(n,p)$
enters as a lemma in service of the main theorems A and B, not as an
end in itself. Finally, all three are shown to be \emph{structural}:
they hold verbatim, with explicit closed forms, for the whole
one-parameter weight family $w_B(S)=B^{\,B^{\,n-|S|}}-1$,
$B\in\mathbb Z_{\ge2}$ (Theorem~\ref{thm:generalB}), so no constant
above is an artefact of the base $B=2$.

\begin{remark}[Geometric realisation]\label{rem:hypercube}
Identifying each $S\in V_n$ with its indicator vector, $V_n$ is exactly
the set of nonzero vertices of the hypercube $Q_n=\{0,1\}^n$, and the
covering relation used throughout \S\ref{sec:prelim}--\S\ref{sec:main-proof}
is exactly the edge set of $Q_n$. Under this identification,
$\Star(i)=\{x\in Q_n:x_i=1\}$ is precisely one of the $2n$ facets of
$Q_n$ (Figures~\ref{fig:cube} and~\ref{fig:tesseract}), and the
complementary-pair matching that drives the argument of
\S\ref{sec:main-proof} -- $S\leftrightarrow[n]\setminus S$ -- is exactly
the antipodal matching on $Q_n$. Nothing in the proofs depends on this
identification; it is recorded because
Figures~\ref{fig:cube}--\ref{fig:tesseract} are drawn in these terms.
\end{remark}

\subsection{Extremal theorem: a known first level}
\label{sec:extremal-input}

The combinatorial input for our main results is an extremal theorem
for the $p$-biased measure on $D_n$, whose first level is essentially
known (\S\ref{sec:known}); it feeds the leading
exponent of \S\ref{sec:main-result}, and the exact prefactor there
additionally uses the second-level theorem of \S\ref{sec:secondlevel}. For
$0<p<1$, $q=1-p$, put $\lambda:=p/q$ and note that
\begin{equation}\label{eq:mu-def}
\mu_p(\mathcal F)\;:=\;\sum_{S\in\mathcal F}p^{|S|}q^{\,n-|S|}
\;=\;q^n\sum_{S\in\mathcal F}\lambda^{|S|},
\end{equation}
so that, up to the overall factor $q^n$, $\mu_p(\mathcal F)$ is an
additive $p$-biased layer functional of $\mathcal F$ (a sum over its
members, not the multiplicative hard-core weight
$\prod_{S}z_S$). At $p=\tfrac13$, after multiplication by $3^n$, it
becomes the exponent $\Lambda$ of the multiplicative sector partition
function of \S\ref{sec:main-result}.
Following Tokushige~\cite{Tokushige2024} write
\[
M_2(n,p)\;:=\;\max\{\mu_p(\mathcal F):\mathcal F\subseteq V_n\text{ is
non-trivial}\},
\]
the case $r=2$ of the general quantity $M_r^t(n,p)$ for non-trivial
$r$-wise $t$-intersecting families. (For $n\le2$ there is no
non-trivial intersecting family, so we assume $n\ge3$ throughout.) For
$p<\tfrac12$, i.e.\ $\lambda<1$,
it has long been known that $M_2(n,p)<p=\mu_p(\Star(i))$: trivial
configurations always dominate, but by a margin that must be quantified.
The value of this margin is not new as an object -- it is a known
extreme point of the profile polytope of non-trivial intersecting
families~\cite{EFK1985,Gerbner2021}, see \S\ref{sec:known} -- and we
state our contribution accordingly: an elementary, self-contained
proof of the closed form, and the complete classification of the
maximisers at every finite $n$.

\begin{theorem}[Extremal exponent]\label{thm:main-mu}
Let $n\ge 3$ and $0<p\le\tfrac12$. Then
\begin{equation}\label{eq:M2}
M_2(n,p)\;=\;p-pq^{\,n-1}+qp^{\,n-1},
\end{equation}
attained by the one-flip stars
\begin{equation}\label{eq:oneflip}
\mathcal F_a^{*}\;:=\;\bigl(\Star(a)\setminus\{\{a\}\}\bigr)
\cup\{[n]\setminus\{a\}\},\qquad a\in[n].
\end{equation}
Moreover, for $0<p<\tfrac12$:
\begin{enumerate}[label=(\roman*),leftmargin=*]
\item if $n\ge 5$, the maximisers are exactly the $n$ families
$\mathcal F_a^{*}$, and these $n$ families are pairwise distinct;
\item if $n=4$, there are exactly eight maximisers: the four one-flip
stars and the four \emph{triangle systems}
$\mathcal T_B=\binom{B}{2}\cup\binom{[4]}{3}\cup\{[4]\}$,
$B\in\binom{[4]}{3}$;
\item if $n=3$, the unique maximiser is the (single) triangle system,
which coincides with every $\mathcal F_a^{*}$.
\end{enumerate}
At $p=\tfrac12$ (critical activity $\lambda=1$) the bound \eqref{eq:M2}
equals $\tfrac12$ and is attained exactly by the non-condensed maximal
independent sets, and by no other non-condensed family: the energy gap
of \S\ref{sec:main-result} closes entirely at criticality.
\end{theorem}

\begin{corollary}[$Q$-ary weight exponents]\label{cor:qary}
Let $Q\ge 2$ be an integer and define the $Q$-ary weight exponent of a
family by $\Lambda_Q(\mathcal F):=\sum_{S\in\mathcal F}(Q-1)^{n-|S|}$.
Then for every non-condensed $\mathcal F$ on $[n]$, $n\ge3$,
\[
\Lambda_Q(\mathcal F)\;\le\;Q^{n-1}-(Q-1)^{n-1}+(Q-1),
\]
with equality for the one-flip stars. For $Q\ge 3$ the equality cases
are exactly those of Theorem~\ref{thm:main-mu} at $p=1/Q<\tfrac12$; for
$Q=2$ every non-condensed maximal independent set attains the bound.
(Take $p=1/Q$ in Theorem~\ref{thm:main-mu} and multiply by $Q^n$.) At
$Q=3$ this is the exponent $3^{n-1}-2^{n-1}+2$ used in
\S\ref{sec:main-result}. For integer $Q\ge3$ both the bound and the equality classification
are, via the embedding $\Phi_Q$ of \S\ref{sec:known}, exactly the case
$r=n$ of Borg's Hilton--Milner theorem for signed
sets~\cite[Theorem~2.3]{Borg2013} (restated in the multi-part setting
in~\cite[Theorem~2]{KwanSudakovVieira2018}, and the $t=1$ case
of~\cite{FranklNie2026,HouHu2026}); the corollary is thus a
Boolean-weighted restatement of a known theorem, and the route through
Theorem~\ref{thm:main-mu} an independent elementary proof of it. The
same product-to-Boolean-lattice reduction appears, for arbitrary part
sizes, as~\cite[Theorem~1.2]{HouHu2026}; at four coordinates their
reduction exhibits the star--triangle dichotomy (their Theorem~4.1)
seen here at $n=4$.
\end{corollary}

The proof of Theorem~\ref{thm:main-mu} (\S\ref{sec:main-proof}) reduces
the problem to \emph{maximal} independent sets of $D_n$ -- maximal
linked systems, in the terminology of \S\ref{sec:prelim} -- and compares
any non-condensed one with a condensed sector pair-by-pair on the
complementary partition of $V_n$, using the Erd\H os--Ko--Rado
theorem~\cite{EKR1961} on each uniform layer. The argument is
elementary: it uses neither the polytope/Kruskal--Katona machinery
discussed next nor stability results~\cite{Friedgut2008}.

\subsection{Main result: the weighted independent-set polynomial
and its leading suppression exponent}\label{sec:main-result}

For the site-dependent, doubly exponential weight
\begin{equation}\label{eq:weight}
w(S)\;:=\;2^{2^{n-|S|}}-1\qquad(S\in V_n),
\end{equation}
write $W(n):=Z(D_n,w)$; this grows doubly exponentially as $|S|$
decreases, so the lower layers carry overwhelmingly the largest weights
and are the most favoured. Decompose
$W(n)=Z_\cap(n)+R(n)$, where $Z_\cap(n)$ collects the trivial
(kernel-bearing) families and $R(n)$ the non-trivial (kernel-free)
ones.

\begin{figure}[t]
\centering
\includegraphics[width=0.82\textwidth]{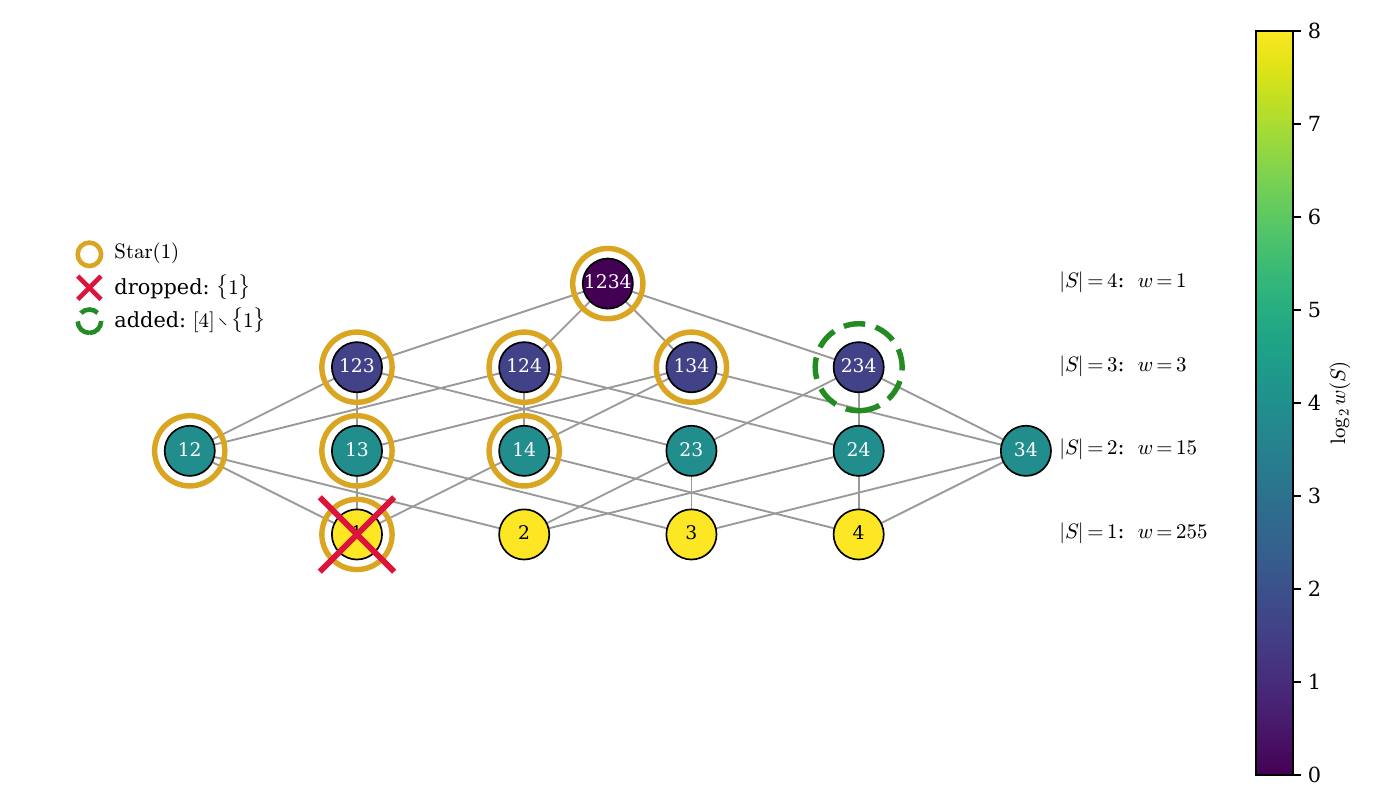}
\caption{The weight~\eqref{eq:weight} on $V_4=2^{[4]}\setminus\{\varnothing\}$,
and the one-flip star~\eqref{eq:oneflip} that attains
Theorem~\ref{thm:main-mu}. Node colour encodes $\log_2 w(S)$; the four
layer weights ($255,15,3,1$) show the doubly exponential decay
directly. Gold ring: the star $\Star(1)$. The one-flip modification --
drop $\{1\}$ (red cross), add $[n]\setminus\{1\}$ (green ring) -- is
the cheapest way to destroy the common element while losing the least
weight.}
\label{fig:hasse}
\end{figure}

\begin{theorem}[Leading suppression exponent]\label{thm:main}
$W(n)\sim n\cdot 2^{3^{n-1}}$ as $n\to\infty$. More precisely,
\begin{equation}\label{eq:Zcap-formula}
Z_\cap(n)\;=\;\sum_{j=1}^{n}(-1)^{j+1}\binom{n}{j}\,2^{3^{n-j}}
\;\sim\;n\cdot 2^{3^{n-1}},
\end{equation}
and the fraction of the partition function carried by non-condensed
sectors obeys the unconditional, non-asymptotic upper bound
\begin{equation}\label{eq:rate}
\frac{R(n)}{Z_\cap(n)}\;\le\;
\frac{\lambda(n)}{n-1}\cdot 2^{-(2^{n-1}-2)}
\qquad(n\ge3),
\end{equation}
where $\lambda(n)$ is the number of non-condensed maximal independent
sets of $D_n$, together with the matching lower bound
\begin{equation}\label{eq:rate-lower}
\frac{R(n)}{Z_\cap(n)}\;\ge\;\Bigl(\frac34-o(1)\Bigr)\cdot
2^{-(2^{n-1}-2)}\qquad(n\to\infty).
\end{equation}
Since $\log_2\lambda(n)=O(2^n n^{-1/2})=o(2^{n-1})$ (Lemma~\ref{lem:MLS-count}),
\eqref{eq:rate} and~\eqref{eq:rate-lower} together give the sharp
exponent
\begin{equation}\label{eq:sharp}
\log_2\frac{Z_\cap(n)}{R(n)}\;=\;2^{n-1}\bigl(1+o(1)\bigr),
\end{equation}
so the exponent $2^{n-1}$ in the suppression factor is sharp at the
leading scale. In fact the additive constant is also determined:
Theorem~\ref{thm:prefactor} below sharpens this to
$R(n)=(\tfrac34+o(1))\,n\cdot2^{M_n^{*}}$ and
$\log_2(Z_\cap/R)=2^{n-1}-2+\log_2\tfrac43+o(1)$
(Corollary~\ref{cor:additive}), via the second-level gap of
Theorem~\ref{thm:second}.
\end{theorem}

The upper bound~\eqref{eq:rate} is proved in \S\ref{sec:application}
below by the union bound $R(n)\le\lambda(n)\cdot2^{M_n^{*}}$ over
kernel-free MLS. The lower bound~\eqref{eq:rate-lower} is proved there
too (Propositions~\ref{prop:one-sector}
and~\ref{prop:n-sector}): briefly, forcing the single ``repair'' set
$[n]\setminus\{a\}$ into a subfamily of $\mathcal F_a^{*}$ already
contributes a fraction $3/4$ of $2^{M_n^{*}}$, and summing this over
the $n$ choices of $a$ -- the $n$ contributions overlap by a
doubly-exponentially negligible amount, since
$\Lambda(\mathcal F_a^{*}\cap\mathcal F_b^{*})=3^{n-2}+4\ll M_n^{*}$
for $a\ne b$ -- gives $R(n)\ge\bigl(\tfrac34-o(1)\bigr)n\cdot2^{M_n^{*}}$.

\begin{corollary}[Gibbs concentration]\label{cor:gibbs}
Let $\mathbb P_n$ denote the Gibbs distribution of the hard-core gas
$(D_n,w)$, i.e.\ $\mathbb P_n(\mathcal F):=\prod_{S\in\mathcal
F}w(S)/W(n)$ on independent sets $\mathcal F$ of $D_n$. Then
\[
\mathbb P_n(\mathcal F\text{ is non-condensed})\;=\;\frac{R(n)}{W(n)}
\;=\;2^{-2^{n-1}(1+o(1))},
\]
i.e.\ a $(D_n,w)$-distributed random independent set is condensed --
contained in some single $\Star(i)$ -- with probability
$1-2^{-2^{n-1}(1+o(1))}\to1$. The sharper constant, $\mathbb
P_n(\text{non-condensed})=(3+o(1))2^{-2^{n-1}}$, follows from the exact
prefactor and is recorded in Corollary~\ref{cor:gibbs-sharp}.
\end{corollary}

The mechanism behind Theorem~\ref{thm:main} is a first-moment
(union-bound) competition between a weight gap and a count. Every
maximal independent set $\mathcal M$ of $D_n$ generates a total weight
$\prod_{S\in\mathcal M}(1+w(S))=2^{\Lambda(\mathcal M)}$ over its own
subfamilies, where $\Lambda(\mathcal M):=\sum_{S\in\mathcal
M}2^{n-|S|}$; the $n$ stars $\Star(i)$ each carry the maximal possible
exponent $\Lambda(\Star(i))=3^{n-1}$, while \emph{every} kernel-free
maximal independent set carries at most $3^{n-1}-2^{n-1}+2$
(Corollary~\ref{cor:qary} at $Q=3$; Figure~\ref{fig:hasse} shows this
gap concretely for $n=4$), and the number $\lambda(n)$ of kernel-free
maximal independent sets is bounded above by a Dedekind number,
$2^{O(2^n/\sqrt n)}$, so the count can never compensate the gap. Using
the exact values of $\lambda(n)$~\cite{BMMV2013,PawelskiSzepietowski2023},
the resulting suppression bound is already decisive at small $n$:
$\ge2^{4.58}$ at $n=4$, $\ge2^{20.9}$ at $n=6$, $\ge2^{188}$ at $n=9$. 


\subsection{Relation to prior work}\label{sec:known}

The quantity $M_2(n,p)$ is not new as an \emph{object}: its value is a
known extreme point of the profile polytope of non-trivial intersecting
families. Erd\H os, Frankl and Katona~\cite{EFK1985} determined the
extreme points of the profile polytope of \emph{all} intersecting
families; the one relevant here (their vector $a_2$, in the notation
used by~\cite{Gerbner2021}) is exactly the profile of $\mathcal
F_a^{*}$. Gerbner~\cite{Gerbner2021} determined, more generally, the
extreme points of the profile polytope of \emph{non-trivial} intersecting
families, which contain the same vector as a member of his family
$\Gamma_c$; combined with his case analysis (his Case~2c1), this
identifies $\mathcal F_a^{*}$ as the correct candidate for \emph{any}
fugacity profile $w(S)=q^n\lambda^{|S|}$ with $\lambda<1$ -- i.e.\ for
every $\mu_p$, $p<\tfrac12$ -- in a few lines on top of his general
theorem. What Theorem~\ref{thm:main-mu} retains here is (a)~a short elementary
proof that bypasses the polytope and Kruskal--Katona machinery
entirely, (b)~a single argument uniform over all $0<p\le\tfrac12$, and (c)~its
use, via Corollary~\ref{cor:qary}, as the gap input to
Theorem~\ref{thm:main}. Concerning~(b) we make the comparison precise
rather than merely negative: the polytope route identifies the
extremal \emph{profile vector} for each weight, so it yields all
maximisers only up to families sharing that profile; at $n=4$ the
one-flip star and the triangle system $\mathcal T_B$ have the
\emph{identical} profile $(0,0,3,4,1)$ (Remark~\ref{rem:AKball}), so
the eight-maximiser statement of Theorem~\ref{thm:main-mu}(ii), and
likewise the $n\ge5$ uniqueness statement, require an argument at the
level of families, not profiles -- which is what
\S\ref{sec:main-proof} supplies. The finite-$n$ maximiser list itself is \emph{not} new: at the
rational points $p=1/Q$, $Q\ge3$, it is contained in Borg's
signed-sets theorem, as explained two paragraphs below.

The systematic study of the size version of this problem goes back to
Brace and Daykin~\cite{BraceDaykin1971}, who determined the maximiser at
$p=\tfrac12$ for the $r$-wise generalisation, $r\ge2$; Frankl and
Tokushige~\cite{FranklTokushige2003,FranklTokushige2006} initiated the
weighted multiply-intersecting line, and
Tokushige~\cite{Tokushige2008,Tokushige2024} extended it to
$M_r(n,p)=\mu_p(\mathrm{BD}_r(n))$ for $r\ge8$ near $p=\tfrac12$, and
recently to the whole range $\tfrac13\le p\le\tfrac12$ for $r=3$, by
linear programming. For $r=2$, the strict inequality $M_2(n,p)<p$ for
$p<\tfrac12$ follows from the uniqueness part of the biased
Erd\H os--Ko--Rado inequality (see~\cite{Friedgut2008,Filmus2017}); a weighted
theorem of Bey and Engel~\cite{BeyEngel2000}, restated in $\mu_p$ form
and applied by Hou and Hu~\cite[Theorem~2.5]{HouHu2026}, bounds
$M_2(n,p)$ implicitly by the maximum of a Hilton--Milner-type term and a
family of Ahlswede--Khachatrian balls, without resolving which candidate
is the larger one; Theorem~\ref{thm:main-mu} resolves this (the
exceptional tie at $n=4$ is with the Ahlswede--Khachatrian ball
$\mathcal S_1$, Remark~\ref{rem:AKball}). The strongest predecessors of Theorem~\ref{thm:main-mu}, however, lie
in two adjacent literatures, and we state the correspondence precisely.
In the theory of quorum systems, a maximal linked system is (the
up-closure of) a \emph{non-dominated coterie}, and, by self-duality,
$\mu_p(M)$ is exactly the system failure probability $F_p$ studied by
Peleg and Wool~\cite{PelegWool1995}: each ground element fails
independently with probability $p$, and $F_p$ is the probability that
no quorum survives. Under this dictionary the one-flip star $\mathcal
F_a^{*}$ is the up-closure of the \emph{Wheel} coterie of that
literature, whose failure probability is the
polynomial~\eqref{eq:M2}, and $M_2(n,p)$ asks for the worst failure
probability among non-dominated coteries other than the singletons --
a quantity squarely within the programme of~\cite{PelegWool1995},
which studies precisely this failure functional, singles out the
Wheel, and determines the extremal range of $F_p$ over all
non-dominated coteries. Closer still,
Borg~\cite{Borg2013} proved a Hilton--Milner theorem for signed sets
which, in the case relevant here, determines the maximum size of a
non-trivial intersecting family in the product $[Q]^{n}$ (two vectors
intersecting when they agree in some coordinate) for every integer
$Q\ge2$, \emph{together with the equality classification for $Q\ge3$}:
the maximum is $Q^{n-1}-(Q-1)^{n-1}+(Q-1)$, attained for $n\ge5$ only
by his family $N_{n,n,Q}$, at $n=4$ additionally by the triangle
family $T_{4,4,Q}$, the two coinciding at $n=3$
(\cite[Theorem~2.3]{Borg2013} at $r=n$; a Hilton--Milner theorem for
the function lattice at support sizes $k<\ell$ -- a range disjoint
from the full-support case $k=\ell$ used here -- was proved earlier by
Erd\H os, Seress and Sz\'ekely~\cite{ErdosSeressSzekely2005}). The embedding
$\Phi_Q(\mathcal F)=\{x\in[Q]^{n}:\{i:x_i=1\}\in\mathcal F\}$ satisfies
$|\Phi_Q(\mathcal F)|=\Lambda_Q(\mathcal F)=Q^{n}\mu_{1/Q}(\mathcal
F)$ and, for $Q\ge3$, carries kernel-free intersecting families to
non-trivial intersecting families in $[Q]^{n}$, with
$N_{n,n,Q}=\Phi_Q(\mathcal F_1^{*})$ and
$T_{4,4,Q}=\Phi_Q(\mathcal T_{\{1,2,3\}})$; consequently, at every
rational $p=1/Q$ with integer $Q\ge3$, the value of
Theorem~\ref{thm:main-mu} \emph{follows from} Borg's theorem, and its
maximiser list coincides, case by case in $n$, with his equality
classification. Kwan, Sudakov and Vieira~\cite{KwanSudakovVieira2018} later
restated the bound in the multi-part setting (their Theorem~2), noting
explicitly that it is a special case of Borg's; the recent
$t$-intersecting product-model theorems of Frankl and
Nie~\cite{FranklNie2026} and Hou and Hu~\cite{HouHu2026} contain it as
their $t=1$ case, the Ahlswede--Khachatrian ball terms
of~\cite{HouHu2026} being the product-model form of the tie visible
here at $n=4$ (the ball $\mathcal S_1$, Remark~\ref{rem:AKball}) and
of the balls $\mathcal A_T$ at the second level at $n=5$
(Theorem~\ref{thm:second}). Theorem~\ref{thm:main-mu} should therefore
be read not as a new extremal maximum but as a known one, for which we
give a short self-contained proof, uniform over all real
$p\in(0,\tfrac12]$ at once, in the Boolean weighted form that the rest
of the paper consumes; we regard continuous $p$ as a convenience of
this single proof rather than a novelty claim, since the \emph{value}
at every real $p$ is in principle available from the profile-polytope
route discussed above. What we have not found in the
quorum, signed-set or multi-part literatures -- and what the present
paper is about -- is the \emph{second} level: the gap theorem and its
classification (Theorem~\ref{thm:second}), the exact
partition-function prefactor (Theorem~\ref{thm:prefactor}), and the
$B$-parameter prefactor constant $1-B^{-B}$
(Theorem~\ref{thm:generalB}); absence from our search is of course not
a proof of novelty.

On the entropy side, 
$\log_2\lambda(n)\le(1+o(1))\binom{n}{\lfloor
n/2\rfloor}$ -- all that Theorem~\ref{thm:main} needs -- follows from
bounding $\lambda(n)$ by the Dedekind number $D(n)$ (every MLS is an
up-set) and Kleitman's theorem on Dedekind
numbers~\cite{Kleitman1969}, refined by Korshunov~\cite{Korshunov1981}.
The total number of MLS on $[n]$ is the number of self-dual monotone
Boolean functions of $n$ variables (the Hosten--Morris numbers
$4,12,81,2646,\dots$ for $n=3,4,5,6$); removing the $n$ stars leaves
the kernel-free count $\lambda(n)=1,8,76,2640,\dots$ used throughout,
tabulated for small $n$ by Brouwer, Mills, Mills and
Verbeek~\cite{BMMV2013} and
Pawelski--Szepietowski~\cite{PawelskiSzepietowski2023}.
Theorem~\ref{thm:main} uses only the crude Dedekind upper bound above,
not any finer asymptotics of $\lambda(n)$. On
the ``increasing-weight'' side -- complementary to the decreasing
fugacity~\eqref{eq:weight} studied here -- Flower and
Mycroft~\cite{FlowerMycroft2025} recently identified the canonical
maximal left-compressed intersecting families as the unique maximisers
of any increasing weight function on a uniform layer; see also the
typical-structure
results~\cite{BaloghDasDelcourtLiuSharifzadeh2015,FranklKupavskii2018,
BaloghGarciaLiWagner2021} for the unweighted ($w\equiv1$) count of
intersecting families.
Very recently Bai, Gu and Zang~\cite{BaiGuZang2026} extended those
counts to the non-uniform bounded-size setting: for
$n\ge2k+2+2\sqrt{k\log k}$ and $k\to\infty$, the number of intersecting
families on $[n]$ all of whose members have size at most $k$ is
$(n+o(1))\,2^{\sum_{i=1}^{k}\binom{n-1}{i-1}}$, so typical bounded-size
intersecting families are trivial. Those are leading-order counting
statements for the \emph{number} of families, valid when $n$ is large
compared with the maximal member size; the present paper allows all of
$V_n$ (no bound on member sizes, so the case $k=n$ lies far outside
that regime) and proves a \emph{weighted} statement -- an additive
error term for the partition function, together with the exact list of
near-extremal families (Corollary~\ref{cor:additive},
Theorem~\ref{thm:second}). Neither statement implies the other.

\subsection{Structure of the paper}

Section~\ref{sec:prelim} sets up the correspondence between independent
sets of $D_n$ and maximal linked systems. Section~\ref{sec:main-proof}
proves the extremal Theorem~\ref{thm:main-mu} (main result C) and
extracts the rigidity Lemma~\ref{lem:rigidity}.
Section~\ref{sec:application} proves the leading-exponent
Theorem~\ref{thm:main}. Section~\ref{sec:defect} is the core: after a
one-defect family that supplies the extremisers, it proves the
second-level gap Theorem~\ref{thm:second} (main result A), the
$\mu_p$-version Corollary~\ref{cor:second-mu}, the exact prefactor
Theorem~\ref{thm:prefactor} and the additive exponent
Corollary~\ref{cor:additive} (main result B).
Appendix~\ref{app:logic} is a short illustrative remark, without any claim of proof,
on the origin of the weight~\eqref{eq:weight} in three-valued logic.
Appendix~\ref{sec:data} contains exact numerical data: the
decomposition $W(n)=Z_\cap(n)+R(n)$ for $n\le5$ (Table~\ref{tab:data})
and the top of the $\Lambda$-spectrum of kernel-free MLS for $n\le6$
(Table~\ref{tab:spectrum}), all reproducible by the script in
Appendix~\ref{app:code}. Appendix~\ref{sec:supplementary} records a
parity obstruction showing that a naive layer-by-layer bound on the
suppression exponent eventually fails, together with a cross-layer
constraint; these supplementary observations are not used in the main
results. Appendix~\ref{app:figs} contains supplementary illustrations.

\section{Maximal linked systems and complementary pairs}\label{sec:prelim}

The \emph{kernel} of a family is
$\ker(\mathcal F):=\bigcap_{S\in\mathcal F}S$, with the convention
$\ker(\varnothing)=[n]$. A family is \emph{kernel-bearing} (or
\emph{trivial}) if $\ker(\mathcal F)\neq\varnothing$ and
\emph{kernel-free} (\emph{non-trivial}) otherwise.

A \emph{maximal linked system} (MLS) on $[n]$ is a maximal intersecting
family of nonempty subsets of $[n]$. Equivalently, an MLS is an
intersecting family $M$ that contains exactly one set from each
complementary pair $(S,[n]\setminus S)$,
$\varnothing\neq S\subsetneq[n]$, together with $[n]$; hence
$|M|=2^{n-1}$. Indeed, $M$ cannot contain both members of a pair, as
they are disjoint; and if it contained neither $S$ nor $S^{c}$, then by
maximality some $A\in M$ would be disjoint from $S$, i.e.\
$A\subseteq S^{c}$, whence every $B\in M$ meets $S^{c}$ (it meets
$A$), so $M\cup\{S^{c}\}$ is intersecting, contradicting maximality.
In particular every MLS is an up-set: if $S\in M$ and
$S\subseteq T$, then $T$ meets every member of $M$, so $T\in M$ by
maximality. Via the indicator of membership, MLS on $[n]$ are
therefore in bijection with self-dual monotone Boolean functions of
$n$ variables (the total count, including the $n$ stars, is the
Hosten--Morris number, cf.~\cite{PawelskiSzepietowski2023}).

\begin{lemma}\label{lem:stars}
The kernel-bearing MLS on $[n]$ are exactly the $n$ stars $\Star(i)$.
\end{lemma}

\begin{proof}
If $i\in\ker(M)$, then every $S\ni i$ meets every member of $M$, so by
maximality $\Star(i)\subseteq M$; and any $S\in M$ must meet
$\{i\}\in\Star(i)\subseteq M$, so $i\in S$. Hence $M=\Star(i)$.
Conversely each $\Star(i)$ is an MLS with kernel $\{i\}$.
\end{proof}

\begin{lemma}[Reduction]\label{lem:reduction}
For all $n\ge3$ and $p\in(0,1)$,
\[
M_2(n,p)\;=\;\max\{\mu_p(M):M\ \text{is a kernel-free MLS on }[n]\}.
\]
\end{lemma}

\begin{proof}
Every non-trivial intersecting $\mathcal F$ extends to an MLS $M$, and
$\ker(M)\subseteq\ker(\mathcal F)=\varnothing$, so $M$ is kernel-free;
since $\mu_p$ is monotone under adding sets,
$\mu_p(\mathcal F)\le\mu_p(M)$. Conversely every kernel-free MLS is a
non-trivial intersecting family.
\end{proof}

\begin{lemma}[Singleton lemma]\label{lem:singleton}
Let $M$ be a kernel-free MLS on $[n]$. Then $\{i\}\notin M$ for every
$i\in[n]$, and consequently $[n]\setminus\{i\}\in M$ for every
$i\in[n]$.
\end{lemma}

\begin{proof}
If $\{i\}\in M$, then every $S\in M$ meets $\{i\}$, forcing
$i\in\ker(M)$ -- a contradiction, since $M$ is kernel-free. Hence
$\{i\}\notin M$ for all $i$. From each complementary pair
$(\{i\},[n]\setminus\{i\})$ exactly one member lies in $M$, so
$[n]\setminus\{i\}\in M$.
\end{proof}

Finally, we record the value at the star: for any $a$ and $p$,
\begin{equation}\label{eq:star-measure}
\mu_p(\Star(a))=\sum_{T\subseteq[n]\setminus\{a\}}p^{1+|T|}q^{\,n-1-|T|}
=p\,(p+q)^{n-1}=p,
\end{equation}
which is the biased Erd\H os--Ko--Rado
value (see~\cite{Friedgut2008,Filmus2017}):
$\mu_p(\mathcal F)\le p$ for every
intersecting $\mathcal F$ and $p\le\tfrac12$, with equality only for
stars when $p<\tfrac12$. (For non-monotone families the inequality
follows by passing to the up-closure, which preserves the intersecting
property and does not decrease $\mu_p$.)
For the non-uniform product measure $\mu_{\mathbf p}$ the analogous
extremality of stars -- the Suda--Tanaka--Tokushige conjecture -- has
very recently been confirmed, in a stronger $t$-intersecting form, by
Wu and Feng~\cite{WuFeng2026stt}.

\section{Proof of Theorem~\ref{thm:main-mu}}\label{sec:main-proof}

By Lemma~\ref{lem:reduction} it suffices to prove: for every
kernel-free MLS $M$ on $[n]$, $n\ge 3$, and $0<p\le\tfrac12$,
\begin{equation}\label{eq:goal}
\mu_p(M)\;\le\;p-pq^{\,n-1}+qp^{\,n-1},
\end{equation}
together with the classification of equality for $p<\tfrac12$.

\begin{proof}[Proof of the bound \eqref{eq:goal}]
Fix $a\in[n]$ and compare $M$ with $\Star(a)$ pair by pair on the
complementary partition of $2^{[n]}\setminus\{\varnothing,[n]\}$; the
common member $[n]$ contributes $p^n$ to both measures.

\emph{Singleton pairs.}
For the pair $(\{a\},[n]\setminus\{a\})$: $\Star(a)$ contains $\{a\}$,
of measure $pq^{n-1}$, while by Lemma~\ref{lem:singleton} $M$ contains
$[n]\setminus\{a\}$, of measure $p^{n-1}q$. Deficit:
$pq^{n-1}-p^{n-1}q\ (\ge 0$ for $p\le\tfrac12)$.
For each pair $(\{j\},[n]\setminus\{j\})$ with $j\neq a$: both families
contain $[n]\setminus\{j\}$ ($\Star(a)$ because
$a\in[n]\setminus\{j\}$, $M$ by Lemma~\ref{lem:singleton}). No
difference.

\emph{Pairs at layers $2\le k<n/2$.}
For a complementary pair $(S,[n]\setminus S)$ with $|S|=k$, the
$\mu_p$-contribution is $p^{n-k}q^{k}$ if the larger member is chosen
and $p^{k}q^{\,n-k}$ if the smaller one is; define the \emph{gain}
\[
h_k(p)\;:=\;p^{k}q^{\,n-k}-p^{\,n-k}q^{k}
\;=\;p^kq^k\bigl(q^{\,n-2k}-p^{\,n-2k}\bigr)\;>\;0
\quad\text{for }k<n/2,\ p<\tfrac12,
\]
and $h_k(\tfrac12)=0$. Within the pair, $\Star(a)$ chooses the smaller
member exactly when $a\in S$; hence the number of small members chosen
by $\Star(a)$ at layer $k$ is $\binom{n-1}{k-1}$. The small members
chosen by $M$ at layer $k$ form a $k$-uniform intersecting family, so
by the Erd\H os--Ko--Rado theorem~\cite{EKR1961} their number
$f_k:=|M\cap\binom{[n]}{k}|$ satisfies
\[
f_k\;\le\;\binom{n-1}{k-1},\qquad 2\le k<n/2 .
\]

\emph{Middle layer.}
If $n$ is even and $k=n/2$, both members of every pair have measure
$p^{n/2}q^{n/2}$; the choices of $M$ and $\Star(a)$ contribute equally.

\emph{Assembly.}
Summing the differences over all pairs,
\begin{equation}\label{eq:assembly}
\mu_p(\Star(a))-\mu_p(M)
\;=\;\bigl(pq^{\,n-1}-qp^{\,n-1}\bigr)
+\sum_{k=2}^{\lceil n/2\rceil-1}
\Bigl(\binom{n-1}{k-1}-f_k\Bigr)h_k(p)
\;\ge\;pq^{\,n-1}-qp^{\,n-1},
\end{equation}
and \eqref{eq:goal} follows from \eqref{eq:star-measure}.

\emph{Tightness.}
The one-flip star $\mathcal F_a^{*}$ is an MLS with
$\ker(\mathcal F_a^{*})=\varnothing$; it agrees with $\Star(a)$ on every
non-singleton pair and replaces $\{a\}$ by $[n]\setminus\{a\}$, so
$\mu_p(\mathcal F_a^{*})=p-pq^{\,n-1}+qp^{\,n-1}$.
\end{proof}

\begin{proof}[Proof of the uniqueness claims]
First note that, since every nonempty set has strictly positive
$\mu_p$-weight, a non-trivial intersecting family attaining $M_2(n,p)$
must itself be a kernel-free MLS: otherwise its kernel-free MLS
extension provided by Lemma~\ref{lem:reduction} would have strictly
larger measure. It therefore suffices to classify the extremal
kernel-free MLS. Let $0<p<\tfrac12$ and let $M$ be a kernel-free MLS
attaining equality in~\eqref{eq:goal}.

\emph{Case $n\ge 5$.}
Since $h_k(p)>0$ for $2\le k<n/2$, equality in \eqref{eq:assembly}
forces $f_k=\binom{n-1}{k-1}$ for every such $k$; in particular
$f_2=n-1$. As $n>4$, the uniqueness part of the Erd\H os--Ko--Rado
theorem at $k=2$ implies that the $2$-sets of $M$ are exactly
$\{\{b,j\}:j\neq b\}$ for some $b\in[n]$ (for $k=2$ this is
elementary: an intersecting family of $2$-sets of the maximum size
$n-1$ is a star, the only other maximal such family being the triangle,
of size $3<n-1$). Now let $S\in M$ with
$b\notin S$. Then $S$ meets every $\{b,j\}$, so $j\in S$ for all
$j\neq b$, i.e.\ $S\supseteq[n]\setminus\{b\}$, whence
$S=[n]\setminus\{b\}$. Thus every member of $M$ either contains $b$ or
equals $[n]\setminus\{b\}$. By the complementarity property, for each
pair $(S,[n]\setminus S)$ other than the singleton pair at $b$, the
member containing $b$ lies in $M$; and $\{b\}\notin M$ forces
$[n]\setminus\{b\}\in M$. Hence $M=\mathcal F_b^{*}$.

\emph{Case $n=4$.}
By Lemma~\ref{lem:singleton}, $M$ contains the four $3$-sets and
$[4]$, and chooses one $2$-set from each of the three complementary
pairs $\{12,34\}$, $\{13,24\}$, $\{14,23\}$. Any two non-complementary
$2$-subsets of $[4]$ intersect, and every $2$-set meets every $3$-set
($2+3>4$), so \emph{all eight} transversals yield intersecting
families, each maximal, and each kernel-free since already the four
$3$-sets have empty intersection. All eight have the same measure
$4p^3q+p^4+3p^2q^2$, which equals the right-hand side
of~\eqref{eq:goal} at $n=4$ (both equal
$p^2(1+q+q^2)+qp^3$). The eight transversals split into the four with a
common element (giving the four $\mathcal F_a^{*}$) and the four
\emph{triangles} $\binom{B}{2}$, $B\in\binom{[4]}{3}$ (giving
the $\mathcal T_B$). Here the layer range $2\le k<n/2$ is empty and the
middle layer is $\mu_p$-neutral, which is why the
Erd\H os--Ko--Rado rigidity is unavailable and the extremal set is
larger.

\emph{Case $n=3$.}
The only kernel-free MLS on $[3]$ is
$\{\{1,2\},\{1,3\},\{2,3\},\{1,2,3\}\}$, which equals
$\mathcal F_a^{*}$ for every $a$.

\emph{Case $p=\tfrac12$.}
All gains $h_k(\tfrac12)$ and the singleton deficit vanish, so
\eqref{eq:assembly} gives $\mu_{1/2}(M)=\tfrac12$ for \emph{every}
kernel-free MLS; indeed $\mu_{1/2}(M)=|M|/2^{n}=\tfrac12$. By the
positivity argument above, no proper subfamily of an MLS attains the
bound, so these are the only extremal non-trivial intersecting
families.
\end{proof}

\begin{remark}\label{rem:ingredients}
The proof uses only the singleton lemma and the Erd\H os--Ko--Rado
theorem with its uniqueness part at $k=2$. The identification
$\Lambda(\mathcal F)=3^n\mu_{1/3}(\mathcal F)$ of
Section~\ref{sec:application} is not needed; it explains, however, why
the doubly exponential enumeration below is governed exactly by this
extremal problem.
\end{remark}

\begin{remark}\label{rem:AKball}
The triangle systems of Theorem~\ref{thm:main-mu}(ii) are copies of
the Ahlswede--Khachatrian ball:
$\mathcal T_B=\{F\subseteq[4]:|F\cap B|\ge 2\}$, i.e.\ the
Brace--Daykin family $\mathrm{BD}_2$ with window $B$. More generally,
for $1\le i\le\lfloor(n-1)/2\rfloor$ the balls
$\mathcal S_i=\{A:|A\cap[1+2i]|\ge 1+i\}$ are
themselves kernel-free MLS (each complementary pair contributes
exactly one member, by majority in the odd window), so
Theorem~\ref{thm:main-mu} contains the full comparison of the two
candidate types in the Bey--Engel
bound~(\cite{BeyEngel2000},~\cite[Theorem~2.5]{HouHu2026}): for
$0<p<\tfrac12$ the Hilton--Milner-type candidate is strictly larger
for every $n\ge 5$; equality with $\mathcal S_1$ occurs only for
$n\in\{3,4\}$
(at $n=4$ the identity
$\mu_p(\mathcal S_1)=p-pq^{3}+qp^{3}$ holds for all $p$), and at
$p=\tfrac12$ all kernel-free MLS have the same measure.
\end{remark}

\section{The weighted independent-set polynomial: proof of Theorem~\ref{thm:main}}
\label{sec:application}

\subsection{The kernel contribution}

\begin{proposition}\label{prop:Zcap}
The kernel-bearing contribution to $W(n)$ satisfies
\begin{equation*}
Z_\cap(n)\;:=\;\sum_{\substack{\mathcal F\in\Intn\\\ker(\mathcal F)\neq\varnothing}}
\prod_{S\in\mathcal F}w(S)
\;=\;\sum_{j=1}^{n}(-1)^{j+1}\binom{n}{j}\,2^{3^{n-j}}
\;\sim\; n\cdot 2^{3^{n-1}}.
\end{equation*}
\end{proposition}

\begin{proof}
For nonempty $J\subseteq[n]$, every subfamily of
$\{S\subseteq[n]:J\subseteq S,\;S\neq\varnothing\}$ is intersecting,
and its weighted count is
\[
\prod_{S\supseteq J}(1+w(S))
\;=\;\prod_{S\supseteq J}2^{2^{n-|S|}}
\;=\;2^{\sum_{k=|J|}^{n}\binom{n-|J|}{k-|J|}\,2^{n-k}}
\;=\;2^{3^{n-|J|}},
\]
the last equality by the binomial identity
$\sum_{t=0}^{m}\binom{m}{t}2^{m-t}=3^m$ with $m=n-|J|$.
Inclusion--exclusion over the events
$\{\mathcal F:\ker(\mathcal F)\ni i\}$ for $i\in[n]$ gives the
formula (the empty family, with $\ker(\varnothing)=[n]$ by convention,
lies in every event and is counted once by the alternating sum).
The $j=1$ term dominates: the $j=2$ term satisfies
$\binom{n}{2}\cdot 2^{3^{n-2}}/(n\cdot 2^{3^{n-1}})
=\tfrac{n-1}{2}\cdot 2^{-2\cdot 3^{n-2}}\to 0$. Indeed the whole
tail is negligible: $\sum_{j=2}^{n}\binom nj 2^{3^{n-j}}\le
2^{n}\,2^{3^{n-2}}=o\!\left(n\,2^{3^{n-1}}\right)$, whence
$Z_\cap(n)\sim n\cdot2^{3^{n-1}}$.
Moreover, by the Bonferroni inequality for this weighted union,
\begin{equation}\label{eq:bonferroni}
Z_\cap(n)\;\ge\;n\cdot 2^{3^{n-1}}-\binom{n}{2}\,2^{3^{n-2}}
\;\ge\;(n-1)\cdot 2^{3^{n-1}}
\qquad(n\ge 2),
\end{equation}
the last step because $\binom{n}{2}\le 2^{2\cdot 3^{n-2}}$ for all
$n\ge 2$.
\end{proof}

\subsection{The weight exponent and the count of maximal linked systems}

The \emph{weight exponent} of an intersecting family is
\begin{equation}\label{eq:Lambda}
\Lambda(\mathcal F)\;:=\;\sum_{S\in\mathcal F}2^{n-|S|}
\;=\;3^n\,\mu_{1/3}(\mathcal F),
\end{equation}
the identity holding because
$3^n(1/3)^{|S|}(2/3)^{n-|S|}=2^{n-|S|}$. Since
$w(S)\le 2^{2^{n-|S|}}$, the product weight of any family is at most
$2^{\Lambda(\mathcal F)}$. Corollary~\ref{cor:qary} with $Q=3$ (i.e.\
Theorem~\ref{thm:main-mu} at $p=\tfrac13$) gives the sharp bound for
the kernel-free maximum:
\begin{equation}\label{eq:M-star}
M_n^{*}\;:=\;\max\{\Lambda(M):M\ \text{kernel-free MLS on }[n]\}
\;=\;3^{n-1}-2^{n-1}+2 .
\end{equation}

Let $\lambda(n)$ denote the number of kernel-free MLS on $[n]$. By
Lemma~\ref{lem:stars}, $\lambda(n)$ equals the total number of MLS
minus $n$.

\begin{lemma}\label{lem:MLS-count}
$\log_2\lambda(n)\;\le\;(1+o(1))\binom{n}{\lfloor n/2\rfloor}
\;=\;O\!\bigl(2^{n}n^{-1/2}\bigr)\;=\;o\!\bigl(2^{n-1}\bigr).$
\end{lemma}

\begin{proof}
An MLS, being an up-set, is determined by its set of minimal elements,
which is an antichain in $2^{[n]}$; hence the number of MLS is at most
the number of antichains, i.e.\ the Dedekind number $D(n)$. By
Kleitman's theorem~\cite{Kleitman1969}
(refined by Korshunov~\cite{Korshunov1981}),
$\log_2 D(n)=(1+o(1))\binom{n}{\lfloor n/2\rfloor}$, and
$\binom{n}{\lfloor n/2\rfloor}=\Theta(2^n n^{-1/2})=o(2^{n-1})$.
\end{proof}

\begin{remark}\label{rem:MLS-exact}
Much more precise information is available. MLS on $[n]$ correspond
bijectively to self-dual monotone Boolean functions of $n$ variables;
their exact numbers for small $n$ were computed by Brouwer, Mills, Mills and
Verbeek~\cite{BMMV2013} (e.g.\ there are $81$ MLS on $[5]$,
of which $76$ are kernel-free), with recent extensions
in~\cite{PawelskiSzepietowski2023}. Only the crude upper bound of
Lemma~\ref{lem:MLS-count} is needed below.
\end{remark}

\subsection{Proof of Theorem~\ref{thm:main}}

\begin{proposition}[MLS reduction]\label{prop:reduction}
$R(n)\;\le\;\lambda(n)\cdot 2^{M_n^{*}}$.
\end{proposition}

\begin{proof}
Every kernel-free intersecting $\mathcal F$ is contained in a
kernel-free MLS $M$ (Lemma~\ref{lem:reduction}), and the total weight
of all subfamilies of a fixed $M$ is
\[
\prod_{S\in M}(1+w(S))=2^{\Lambda(M)}\le 2^{M_n^{*}}.
\]
Sum over all kernel-free MLS.
\end{proof}

\begin{proposition}[Single-sector lower bound]\label{prop:one-sector}
Let $\Sigma_a:=\sum_{\mathcal F\subseteq\mathcal F_a^{*},\;\ker\mathcal
F=\varnothing}\prod_{S\in\mathcal F}w(S)$. Then
\[
\Sigma_a\;\ge\;\frac34\cdot2^{M_n^{*}}\Bigl[1-(n-1)\,
2^{-(2\cdot3^{n-2}-2^{n-1})}\Bigr],
\]
and, since trivially $\Sigma_a\le T=\tfrac34\cdot2^{M_n^{*}}$ (kernel-free
subfamilies containing $A$ are a subset of all subfamilies containing
$A$), both bounds together give
$\Sigma_a=\tfrac34\cdot2^{M_n^{*}}\bigl(1-o(1)\bigr)$.
\end{proposition}

\begin{proof}
Fix $a=1$ without loss of generality. Write $A:=[n]\setminus\{1\}$,
the unique member of $\mathcal F_1^{*}$ not containing $1$; since
$w(A)=2^{2^{1}}-1=3$, a family $\mathcal F\subseteq\mathcal F_1^{*}$
with $\ker\mathcal F=\varnothing$ must contain $A$ (otherwise every
member contains $1$, so $1\in\ker\mathcal F$). Summing over all
subfamilies containing $A$, kernel-free or not,
\[
T\;:=\;\sum_{\substack{\mathcal F\subseteq\mathcal F_1^{*}\\A\in\mathcal
F}}\prod_{S\in\mathcal F}w(S)
\;=\;w(A)\!\!\prod_{S\in\mathcal F_1^{*}\setminus\{A\}}\!\!(1+w(S))
\;=\;\frac{w(A)}{1+w(A)}\cdot2^{\Lambda(\mathcal F_1^{*})}
\;=\;\frac34\cdot2^{M_n^{*}}.
\]
Every $\mathcal F$ counted in $T$ but not in $\Sigma_1$ has $A\in\mathcal
F$ yet $\ker\mathcal F\ne\varnothing$; since $A\notin\ker\mathcal F$
already, this means $j\in\ker\mathcal F$ for some $j\ne1$, i.e.\
$\mathcal F\cap\mathcal C_j=\varnothing$ for
$\mathcal C_j:=\{S\in\mathcal F_1^{*}: j\notin S\}$. A direct count
(sets containing $1$, not $j$: parametrised by
$T'\subseteq[n]\setminus\{1,j\}$, giving $\Lambda=2\cdot3^{n-2}$ minus
the excluded singleton $\{1\}$'s contribution $2^{n-1}$) gives
$\Lambda(\mathcal C_j)=2\cdot3^{n-2}-2^{n-1}$, and, exactly as for $T$
(the events ``$A\in\mathcal F$'' and ``$\mathcal F\cap\mathcal
C_j=\varnothing$'' involve disjoint parts of $\mathcal F_1^{*}$, since
$A\notin\mathcal C_j$),
\[
\sum_{\substack{\mathcal F\subseteq\mathcal F_1^{*}\\A\in\mathcal
F,\,\mathcal F\cap\mathcal C_j=\varnothing}}\prod_{S\in\mathcal
F}w(S)\;=\;\frac34\cdot2^{M_n^{*}-\Lambda(\mathcal C_j)}.
\]
Union-bounding over the $n-1$ choices of $j\ne1$ gives the stated
lower bound.
\end{proof}

\begin{proposition}[Extension to all $n$ sectors]\label{prop:n-sector}
\[
R(n)\;\ge\;\sum_{a=1}^{n}\Sigma_a-\binom n2 2^{3^{n-2}+4}
\;=\;\Bigl(\tfrac34-o(1)\Bigr)\,n\cdot2^{M_n^{*}}.
\]
\end{proposition}

\begin{proof}
For $a\ne b$, $\mathcal F_a^{*}\cap\mathcal F_b^{*}=\{S:\{a,b\}\subseteq
S\}\cup\{[n]\setminus\{a\}\}\cup\{[n]\setminus\{b\}\}$ (the first part
since $S\in\mathcal F_a^{*}\cap\mathcal F_b^{*}$ with $S\ne[n]\setminus
a,[n]\setminus b$ forces $a,b\in S$; the two co-singletons lie in both families for $n\ge3$:
$[n]\setminus\{a\}\in\mathcal F_a^{*}$ by definition and, containing
$b$ and having size $n-1\ge2$, it lies in
$\Star(b)\setminus\{\{b\}\}\subseteq\mathcal F_b^{*}$; symmetrically
for $[n]\setminus\{b\}$), a disjoint union with
\[
\Lambda(\mathcal F_a^{*}\cap\mathcal F_b^{*})
=\underbrace{3^{n-2}}_{\{a,b\}\subseteq S}+\underbrace{2}_{[n]\setminus
a}+\underbrace{2}_{[n]\setminus b}=3^{n-2}+4.
\]
Any family counted more than once in
$\sum_a\Sigma_a$ is a kernel-free subfamily of some
$\mathcal F_a^{*}\cap\mathcal F_b^{*}$, so by inclusion--exclusion
(Bonferroni),
\[
R(n)\;\ge\;\sum_{a=1}^{n}\Sigma_a-\sum_{a<b}\;
\sum_{\mathcal F\subseteq\mathcal F_a^{*}\cap\mathcal F_b^{*}}
\prod_{S\in\mathcal F}w(S)
\;\ge\;n\Sigma_1-\binom n2\,2^{\Lambda(\mathcal F_a^{*}\cap\mathcal
F_b^{*})}
\;=\;n\Sigma_1-\binom n2\,2^{3^{n-2}+4}.
\]
Since $3^{n-2}\gg2^{n-1}$, the subtracted term is smaller than
$n\Sigma_1\sim\tfrac34 n\cdot2^{M_n^{*}}$ by a doubly-exponential
factor, giving the stated asymptotic form.
\end{proof}

\begin{remark}[Numerical sanity checks]\label{rem:sanity}
The quantities in Propositions~\ref{prop:one-sector}
and~\ref{prop:n-sector} were checked against exhaustive enumeration
for $n=4,5$ (Appendix~\ref{app:code}): at $n=4$,
$\Sigma_1=1\,568\,322$, so $\Sigma_1/T=0.99711\ldots$ against the
predicted lower bound $1-(n-1)2^{-(2\cdot3^{n-2}-2^{n-1})}
=0.99707\ldots$ (the small residual is expected, as the union bound
need not be exact); at $n=5$ the ratio is $1.00000\ldots$ to the
precision computed; $\Lambda(\mathcal F_a^{*}\cap\mathcal
F_b^{*})=13=3^{2}+4$ for all six pairs at $n=4$; and
$n\Sigma_1=6\,273\,288\le R(4)=13\,158\,882$
(Table~\ref{tab:data}). These computations play no role in the
proofs.
\end{remark}

\begin{proof}[Proof of Theorem~\ref{thm:main}]
Proposition~\ref{prop:Zcap} gives $Z_\cap(n)\sim n\cdot 2^{3^{n-1}}$ and,
by~\eqref{eq:bonferroni},
$Z_\cap(n)\ge(n-1)\cdot 2^{3^{n-1}}$. Combining
Proposition~\ref{prop:reduction} with~\eqref{eq:M-star},
\[
\frac{R(n)}{Z_\cap(n)}\;\le\;
\frac{\lambda(n)\cdot 2^{\,3^{n-1}-2^{n-1}+2}}
{(n-1)\cdot 2^{3^{n-1}}}
\;=\;\frac{\lambda(n)}{n-1}\cdot 2^{-(2^{n-1}-2)},
\]
an unconditional, non-asymptotic inequality for every $n\ge 3$,
proving~\eqref{eq:rate}. By Lemma~\ref{lem:MLS-count},
$\log_2\lambda(n)=O(2^n n^{-1/2})$, so the exponent of the right-hand
side is $-(2^{n-1}-2)+O(2^n n^{-1/2})=-2^{n-1}(1-O(n^{-1/2}))$. For the
lower bound, Proposition~\ref{prop:n-sector} and
$Z_\cap(n)\le n\cdot2^{3^{n-1}}(1+o(1))$ give
\[
\frac{R(n)}{Z_\cap(n)}\;\ge\;
\frac{\bigl(\tfrac34-o(1)\bigr)n\cdot2^{M_n^{*}}}{n\cdot2^{3^{n-1}}(1+o(1))}
\;=\;\Bigl(\tfrac34-o(1)\Bigr)2^{-(2^{n-1}-2)},
\]
proving~\eqref{eq:rate-lower}. Explicitly,
\[
\log_2\frac{Z_\cap(n)}{R(n)}\;\ge\;
2^{n-1}-2-\log_2\frac{\lambda(n)}{n-1},
\qquad
\log_2\frac{Z_\cap(n)}{R(n)}\;\le\;2^{n-1}-2+O(1),
\]
so both bounds agree with~\eqref{eq:sharp}.
\qedhere
\end{proof}

\section{The second level: near-extremal maximal linked
systems}\label{sec:defect}

Having identified the extremal kernel-free maximal linked systems
(the $n$ one-flip stars $\mathcal F_a^{*}$, uniquely so for $n\ge5$),
we now determine the \emph{second} level exactly: the largest value of
$\Lambda$, equivalently of $\mu_p$, over kernel-free MLS other than
these. This closes the gap that the leading-exponent estimate of
\S\ref{sec:main-result} left open on the near-extremal side, and shows
in particular that there is a genuine \emph{spectral gap} below the
extremal value. We first record the one-defect family that will supply
the extremisers, then prove the second-level theorem.

\subsection{A one-defect family}\label{sec:onedefect}

The extremal family $\mathcal F_a^{*}$ makes the cheapest modification
of $\Star(a)$ that destroys the kernel \emph{within the one-defect
family $U_{i,B}$ defined below}; Theorem~\ref{thm:main-mu} separately
shows it is cheapest overall (uniquely so for $n\ge5$). Within this
one-defect family, the cost is explicit and strictly increasing away
from the one-flip case, uniformly in $p$.

\begin{proposition}[Defect cost]\label{prop:defect}
Let $0<p\le\tfrac12$, $i\in[n]$ and $B\subseteq[n]\setminus\{i\}$ with
$|B|=m\ge 2$. The \emph{one-defect system}
\[
U_{i,B}\;:=\;\{S\ni i:S\cap B\neq\varnothing\}\,\cup\,
\{S\not\ni i:B\subseteq S\}
\]
is a kernel-free MLS, and
\[
\mu_p(U_{i,B})\;=\;p-pq^{m}+qp^{m}.
\]
The defect $\delta_m(p):=pq^{m}-qp^{m}$ satisfies
$\delta_2(p)=\delta_3(p)$ and is strictly decreasing in $m$ on
$\{3,\dots,n-1\}$ when $p<\tfrac12$, with minimum
$\delta_{n-1}(p)=pq^{\,n-1}-qp^{\,n-1}$ attained by the one-flip star
$U_{i,[n]\setminus\{i\}}=\mathcal F_i^{*}$.
\end{proposition}

\begin{proof}
\emph{$U_{i,B}$ is a kernel-free MLS.} Any two members containing $i$
meet at $i$; any two members containing $B$ meet on $B$; a member
$S\ni i$ with $S\cap B\neq\varnothing$ meets any $T\supseteq B$. For
each complementary pair $(S,S^{c})$ with $i\in S$: if
$S\cap B\neq\varnothing$ then $S\in U_{i,B}$ and
$B\not\subseteq S^{c}$, so $S^{c}\notin U_{i,B}$; if
$S\cap B=\varnothing$ then $S\notin U_{i,B}$ and $B\subseteq S^{c}$,
so $S^{c}\in U_{i,B}$. Thus $U_{i,B}$ contains exactly one member of
each pair and is an MLS. Its kernel is contained in
$\{i,b\}\cap\{i,b'\}\cap B=\varnothing$ for distinct $b,b'\in B$.
At $m=n-1$ the system $U_{i,B}$ is the one-flip star
$\mathcal F_i^{*}$ itself -- the up-closure of the Wheel coterie
of~\cite{PelegWool1995}, and the Hilton--Milner-type candidate of
the Bey--Engel bound~(\cite{BeyEngel2000},
\cite[Theorem~2.5]{HouHu2026}); the family $(U_{i,B})_{2\le m\le n-1}$
thus interpolates between near-star perturbations and that candidate.

\emph{Measure.} The surviving part of the star is
\[
\mu_p(\{S\ni i:S\cap B\neq\varnothing\})
=\mu_p(\Star(i))-\mu_p(\{S\ni i:S\subseteq[n]\setminus B\})
=p-p\,q^{m},
\]
since the excluded sets contribute
$p\,q^{m}(p+q)^{n-1-m}=pq^{m}$. The added part contributes
$\mu_p(\{S\not\ni i:B\subseteq S\})=q\,p^{m}(p+q)^{n-1-m}=qp^{m}$.

\emph{Monotonicity.}
$\delta_m-\delta_{m+1}=pq^{m}(1-q)-qp^{m}(1-p)
=p^{2}q^{m}-q^{2}p^{m}=p^{2}q^{2}\bigl(q^{m-2}-p^{m-2}\bigr)$,
which vanishes at $m=2$ and is strictly positive for $m\ge 3$,
$p<\tfrac12$.
\end{proof}

In other words, opening a non-star branch of width $m$ costs
$\delta_m(p)$ in measure. For $n\ge 5$ the cheapest such perturbation
is the one-flip; for $n=4$ the widths $m=2$ and $m=3$ have equal cost
($\delta_2=\delta_3$), in accordance with the degeneracy in
Theorem~\ref{thm:main-mu}(ii). At $p=\tfrac13$ this
recovers, after multiplication by $3^n$, the exponent costs
$(2^m-2)\cdot 3^{n-1-m}$ with the curious coincidence
$\delta_2=\delta_3$.

\subsection{The second-level theorem}\label{sec:secondlevel}

Throughout this subsection write $L_k$ for the $k$-th layer
$\{S\in V_n:|S|=k\}$ and, for a kernel-free MLS $M$, write
$f_k:=|M\cap L_k|$ for its layer profile. We use two facts. First, a
profile identity for $\Lambda$; second, a rigidity statement already
contained in the proof of Theorem~\ref{thm:main-mu}.

\begin{lemma}[Profile identity]\label{lem:profile}
For every MLS $M$ on $[n]$,
\[
\Lambda(M)\;=\;A(n)\;+\;\sum_{1\le k<n/2}f_k\,\bigl(2^{\,n-k}-2^{k}\bigr),
\qquad
A(n):=1+\sum_{1\le k<n/2}\binom nk 2^{k}
+[\,2\mid n\,]\,\tfrac12\binom{n}{n/2}2^{n/2}.
\]
\end{lemma}

\begin{proof}
$M$ contains exactly one of $S,S^{c}$ from each complementary pair and
also $[n]$. Pair each $S\in L_k$ ($k<n/2$) with $S^{c}\in L_{n-k}$: the
pair contributes $2^{\,n-k}$ if $M$ takes the smaller side $S$ and
$2^{k}$ if it takes $S^{c}$, i.e.\ a profile-dependent
excess $2^{\,n-k}-2^{k}$ exactly when $S\in M$. Summing the baseline
$2^{k}$ over \emph{all} pairs gives, together with the self-complementary
middle layer at even $n$ (each of its $\tfrac12\binom{n}{n/2}$ pairs
contributing $2^{n/2}$) and the term $1$ for $[n]$, the constant
$A(n)$; the profile-dependent excess is
$\sum_{k<n/2}f_k(2^{\,n-k}-2^{k})$. Every value used is verified against
direct enumeration in Appendix~\ref{app:code}.
\end{proof}

\begin{lemma}[Rigidity at layer two]\label{lem:rigidity}
Let $n\ge5$ and let $M$ be a kernel-free MLS on $[n]$ with $f_2=n-1$.
Then $M=\mathcal F_a^{*}$ for some $a\in[n]$. Consequently, if $M$ is a
kernel-free MLS with $M\neq\mathcal F_a^{*}$ for all $a$, then
$f_2\le n-2$.
\end{lemma}

\begin{proof}
This is exactly the argument of Case $n\ge5$ in the proof of
Theorem~\ref{thm:main-mu}, which uses only the hypothesis $f_2=n-1$:
the $n-1$ two-sets of $M$ form an intersecting family of $2$-sets of
the maximum size $n-1=\binom{n-1}{1}$, hence (as $n>4$, by the
uniqueness in Erd\H os--Ko--Rado) a star $\{\{a,j\}:j\ne a\}$; any $S\in M$
with $a\notin S$ then satisfies $S\supseteq[n]\setminus\{a\}$, forcing
$S=[n]\setminus\{a\}$, and complementarity together with
$\{a\}\notin M$ (Lemma~\ref{lem:singleton}) gives $M=\mathcal F_a^{*}$.
\end{proof}

\begin{lemma}[Reconstruction from a punctured star]\label{lem:reconstruct}
Let $n\ge5$ and let $M$ be a kernel-free MLS on $[n]$ whose $2$-sets
are exactly $\{\{a,j\}:j\in[n]\setminus\{a,c\}\}$ for some $a\in[n]$
and one $c\ne a$. Then $M=U_{a,B}$ with $B=[n]\setminus\{a,c\}$,
$|B|=n-2$. (No hypothesis on the higher layers is needed: the $2$-set
structure and maximality already determine $M$.)
\end{lemma}

\begin{proof}
Write $B=[n]\setminus\{a,c\}$. Let $S\in M$ with $a\notin S$. Then $S$
meets every present $2$-set $\{a,j\}$ ($j\in B$), so $j\in S$ for all
$j\in B$, i.e.\ $B\subseteq S$; thus every $a$-free member of $M$
contains $B$, thus lies in $\{B,\,B\cup\{c\}\}=\{[n]\setminus\{a,c\},\,
[n]\setminus\{a\}\}$. Conversely consider the complementary pairs. For
a pair $(S,S^{c})$ with $a\in S$: if $S\cap B\neq\varnothing$ then $S$
meets every member of $M$ containing $a$ (at $a$) and every $a$-free
member (in $B$, since those contain $B$), so maximality places
$S\in M$; if $S\cap B=\varnothing$, i.e.\ $S\subseteq\{a,c\}$, then
$S\in\{\{a\},\{a,c\}\}$, and $\{a\}\notin M$ (Lemma~\ref{lem:singleton})
while $\{a,c\}\notin M$ (it is not among the listed $2$-sets), so
$S^{c}=[n]\setminus S\in M$; here $S^{c}\supseteq B$ and $a\notin
S^{c}$, consistent with the previous paragraph. Collecting: $M$
consists of all $S\ni a$ with $S\cap B\neq\varnothing$, together with
$B$ and $[n]\setminus\{a\}=B\cup\{c\}$, which is exactly
$U_{a,B}$.
\end{proof}

\begin{theorem}[Second-level gap]\label{thm:second}
Let $n\ge5$ and let $M$ be a kernel-free MLS on $[n]$ with
$M\neq\mathcal F_a^{*}$ for every $a\in[n]$. Then
\[
\Lambda(M)\;\le\;M_n^{*}-\bigl(2^{\,n-2}-4\bigr)
\;=\;3^{\,n-1}-3\cdot2^{\,n-2}+6,
\]
and this bound is attained. Equality holds precisely for
\begin{enumerate}[label=\textup{(\alph*)},leftmargin=*]
\item the $n(n-1)$ systems $U_{a,B}$ with $|B|=n-2$
{\upshape(Proposition~\ref{prop:defect})}, for every $n\ge5$; and
\item when $n=5$ only, in addition the $10$ Ahlswede--Khachatrian balls
$\mathcal A_T=\{S:|S\cap T|\ge2\}$, $T\in\binom{[5]}{3}$.
\end{enumerate}
In particular there is no kernel-free MLS whose $\Lambda$-value lies
strictly between the second level $3^{\,n-1}-3\cdot2^{\,n-2}+6$ and the
maximum $M_n^{*}=3^{\,n-1}-2^{\,n-1}+2$; the two differ by
$2^{\,n-2}-4$.
\end{theorem}

\begin{proof}
By Lemma~\ref{lem:singleton}, $f_1=0$. By Lemma~\ref{lem:rigidity},
$f_2\le n-2$. Every coefficient $2^{\,n-k}-2^{k}$ in
Lemma~\ref{lem:profile} is positive for $k<n/2$, and layerwise
Erd\H os--Ko--Rado gives $f_k\le\binom{n-1}{k-1}$ for $2\le k<n/2$
(the $k$-sets of $M$ form a $k$-uniform intersecting family). Hence
\[
\Lambda(M)
=A(n)+\sum_{2\le k<n/2}f_k\bigl(2^{\,n-k}-2^{k}\bigr)
\le A(n)+(n-2)\bigl(2^{\,n-2}-4\bigr)
+\sum_{3\le k<n/2}\binom{n-1}{k-1}\bigl(2^{\,n-k}-2^{k}\bigr).
\]
The right-hand side equals $M_n^{*}-(2^{\,n-2}-4)$: indeed the same
computation with $f_2=n-1$ gives $M_n^{*}$ (this is the extremal case
of Theorem~\ref{thm:main-mu}), and reducing $f_2$ from $n-1$ to $n-2$
subtracts exactly one coefficient $2^{\,n-2}-2^{2}=2^{\,n-2}-4$. This
proves the inequality.

\emph{Equality.} Since $c_k:=2^{\,n-k}-2^{k}>0$ for every
$2\le k<n/2$, dropping any $k$-set below the EKR maximum
$\binom{n-1}{k-1}$ strictly lowers $\Lambda$. Non-extremality forces,
via Lemma~\ref{lem:rigidity}, a deficit of at least one at layer $2$
($f_2\le n-2$); this single forced deficit already costs $c_2$, and any
further deficit at any layer only lowers $\Lambda$ further. Hence
equality requires $f_k=\binom{n-1}{k-1}$ for all $3\le k<n/2$ and
$f_2=n-2$ exactly (a layer-$2$ deficit of exactly one). Note we do
\emph{not} claim $c_2$ is smallest among the $c_k$ -- for $n\ge7$ it is
in fact the largest; the point is only that the layer-$2$ deficit is
\emph{forced} while the others are free to vanish, and every $c_k$ is
positive. Thus the $2$-sets of $M$ form an intersecting family of size
$n-2$, and for $k\ge3$ the $k$-sets are EKR-extremal. An intersecting
family of $2$-sets of size $n-2$ is either a star minus one edge,
$\{\{a,j\}:j\in[n]\setminus\{a,c\}\}$ for some $a$ and one omitted
$c\ne a$, or (only when $n-2=3$, i.e.\ $n=5$) a triangle
$\binom{T}{2}$, $|T|=3$: for $n\ge6$ a pairwise-intersecting family of
$\ge4$ edges cannot be a triangle, hence (by the Hilton--Milner
description of intersecting $2$-graphs) lies in a star, and having
$n-2$ edges is that star with one edge removed.

In the star-minus-edge case, Lemma~\ref{lem:reconstruct} gives
$M=U_{a,B}$ with $B=[n]\setminus\{a,c\}$, $|B|=n-2$; the choice of $a$
and of the omitted $c\ne a$ gives the $n(n-1)$ systems in~(a).

In the triangle case ($n=5$), the three $2$-sets are $\binom{T}{2}$,
$|T|=3$: any $S\in M$ meets each of the three pairs in $\binom{T}{2}$,
which forces $|S\cap T|\ge2$; conversely, in each complementary pair
$(S,S^{c})$ exactly one side has $|\cdot\cap T|\ge2$ (as $|T|=3$, the
two sides split $T$ as $|S\cap T|+|S^{c}\cap T|=3$, so exactly one
exceeds $1$), whence maximality gives $M=\{S:|S\cap T|\ge2\}=\mathcal
A_T$. This yields the $\binom{5}{3}=10$ families in~(b). For $n\ge6$ a
triangle has size $3<n-2$, so this case does not arise, and only~(a)
survives.

All multiplicities and the top of the $\Lambda$-spectra for $n=5,6$ are
confirmed by exhaustive enumeration in Appendix~\ref{app:code}:
$(67^{\times5},63^{\times30},59^{\times30},55^{\times10})$ for $n=5$
and $(213^{\times6},\allowbreak 201^{\times30},\allowbreak 189^{\times140},\allowbreak 177^{\times480})$
for $n=6$, matching $M_5^{*}=67$, second level $63$, and
$M_6^{*}=213$, second level $201$.
\end{proof}

\begin{proposition}[One profile, two families]\label{prop:oneprofile}
Let $n=5$. All $30$ kernel-free MLS attaining the second level of
Theorem~\ref{thm:second} have one and the same layer profile,
\[
(f_1,f_2,f_3,f_4,f_5)=(0,3,7,5,1),
\]
yet they split into two non-isomorphic types: in the $20$ systems
$U_{a,B}$ ($|B|=3$) the three $2$-sets form a star $K_{1,3}$, while in
the $10$ balls $\mathcal A_T$ they form a triangle. Consequently no
function of the layer profile alone -- in particular no
profile-polytope or Kruskal--Katona-type argument -- can separate the
two types or produce the equality classification of
Theorem~\ref{thm:second}.
\end{proposition}

\begin{proof}
By Theorem~\ref{thm:second} the second level is attained exactly by
the $20$ systems $U_{a,B}$ with $|B|=3$ and the $10$ balls
$\mathcal A_T$. Their profiles are computed directly. For $U_{a,B}$
with $B=[5]\setminus\{a,c\}$: no singleton belongs
(Lemma~\ref{lem:singleton}); the $2$-sets are the three sets
$\{a,b\}$, $b\in B$, and they share the element $a$; the $3$-sets
are the six $3$-sets containing $a$ (each meets $B$, since a pair from
$B\cup\{c\}$ cannot lie inside $\{c\}$) together with $B$ itself,
seven in number; the $4$-sets are the four containing $a$ together with
$B\cup\{c\}=[5]\setminus\{a\}$, five in number; and
$[5]\in U_{a,B}$. For $\mathcal A_T=\{S:|S\cap T|\ge2\}$: the
$2$-sets are the three pairs inside $T$, pairwise intersecting with
empty common intersection; the $3$-sets are the
$\binom32\binom21=6$ sets meeting $T$ in exactly two elements
together with $T$ itself, seven in number; every $4$-set meets $T$ in
at least two elements, giving five; and $[5]\in\mathcal A_T$. Both
profiles equal $(0,3,7,5,1)$. The two types are non-isomorphic already
at layer two: a star $K_{1,3}$ has a common element, a triangle does
not. Since every quantity computable from the profile vector takes the
same value on both types, the final claim follows.
\end{proof}

The same profile-deficit argument, applied directly to $\mu_p$ rather
than to $\Lambda$, gives the second level for the biased measure at
every $0<p<\tfrac12$ (not only at $p=\tfrac13$), giving the second
level for the biased measure across the whole sub-critical range.

\begin{corollary}[Second level for $\mu_p$]\label{cor:second-mu}
Let $n\ge5$ and $0<p<\tfrac12$. The largest value of $\mu_p$ over
kernel-free MLS other than the one-flip stars $\mathcal F_a^{*}$ is
\[
\max_{\substack{M\ \mathrm{kernel\text-free\ MLS}\\ M\neq\mathcal
F_a^{*}\ \forall a}}\mu_p(M)
\;=\;p-pq^{\,n-2}+qp^{\,n-2}
\;=\;M_2(n,p)-h_2(p),
\qquad h_2(p)=p^2q^{\,n-2}-p^{\,n-2}q^2,
\]
attained by the systems $U_{a,B}$ with $|B|=n-2$ {\upshape(and, for
$n=5$, additionally by the balls $\mathcal A_T$)}.
\end{corollary}

\begin{proof}
The assembly identity~\eqref{eq:assembly} of \S\ref{sec:main-proof}
gives, for any kernel-free MLS $M$,
\[
\mu_p(\Star(a))-\mu_p(M)
=\bigl(pq^{\,n-1}-qp^{\,n-1}\bigr)
+\sum_{2\le k<n/2}\Bigl(\tbinom{n-1}{k-1}-f_k\Bigr)h_k(p),
\]
where $h_k(p)=p^kq^{\,n-k}-p^{\,n-k}q^k>0$ and the singleton term $pq^{\,n-1}-qp^{\,n-1}$ is the layer-$1$
deficit (present because $f_1=0$ while $\Star(a)$ contains $\{a\}$).
Since $\mu_p(\Star(a))=p$ and, by Theorem~\ref{thm:main-mu},
$M_2(n,p)=\mu_p(\mathcal F_a^{*})=p-(pq^{\,n-1}-qp^{\,n-1})$, this
rearranges to the deficit \emph{from the one-flip star},
\[
M_2(n,p)-\mu_p(M)
\;=\;\sum_{2\le k<n/2}\Bigl(\tbinom{n-1}{k-1}-f_k\Bigr)h_k(p),
\]
in which the singleton term has cancelled. For $M\ne\mathcal F_a^{*}$,
Lemma~\ref{lem:rigidity} forces $f_2\le n-2$, so the layer-$2$ term is
at least $h_2(p)>0$ and every other term is non-negative; hence
\[
\mu_p(M)\;\le\;M_2(n,p)-h_2(p)\;=\;p-pq^{\,n-2}+qp^{\,n-2},
\]
the last equality being the elementary identity
\[
\begin{aligned}
M_2(n,p)-h_2(p)
&=p-pq^{\,n-1}+qp^{\,n-1}-p^2q^{\,n-2}+p^{\,n-2}q^2\\
&=p-pq^{\,n-2}(q+p)+qp^{\,n-2}(p+q)
=p-pq^{\,n-2}+qp^{\,n-2},
\end{aligned}
\]
using $p+q=1$.
Equality holds iff $f_2=n-2$ and $f_k=\binom{n-1}{k-1}$ for $k\ge3$ --
the profile of Theorem~\ref{thm:second}, attained exactly by
$\{U_{a,B}:|B|=n-2\}$ (and the balls $\mathcal A_T$ at $n=5$), with
$\mu_p(U_{a,B})=p-pq^{\,n-2}+qp^{\,n-2}$
(Proposition~\ref{prop:defect} at $m=n-2$).
\end{proof}

The second-level gap also upgrades the leading-exponent estimate of
Theorem~\ref{thm:main} to an \emph{exact prefactor}: the entropy
$\lambda(n)$ can no longer wash out the constant, because every
non-extremal sector is suppressed by the fixed extra factor
$2^{-(2^{n-2}-4)}$.

\begin{theorem}[Exact prefactor]\label{thm:prefactor}
As $n\to\infty$,
\[
R(n)\;=\;\Bigl(\tfrac34+o(1)\Bigr)\,n\cdot2^{M_n^{*}},
\qquad M_n^{*}=3^{\,n-1}-2^{\,n-1}+2.
\]
\end{theorem}

\begin{proof}
Decompose $R(n)=R_{\mathrm{top}}+R_{\mathrm{non\text-top}}$, where
$R_{\mathrm{top}}$ is the weighted total of all kernel-free families
contained in at least one extremal sector $\mathcal F_a^{*}$, and
$R_{\mathrm{non\text-top}}$ the weighted total of those in no
$\mathcal F_a^{*}$.

\emph{Top part.} By definition $R_{\mathrm{top}}\le\sum_{a=1}^{n}
\Sigma_a$, and each $\Sigma_a\le T=\tfrac34\cdot2^{M_n^{*}}$
(Proposition~\ref{prop:one-sector}), so $R_{\mathrm{top}}\le
\tfrac34 n\cdot2^{M_n^{*}}$. In the other direction, the
inclusion--exclusion lower bound of Proposition~\ref{prop:n-sector}
gives $R_{\mathrm{top}}\ge\sum_a\Sigma_a-\binom n2 2^{3^{n-2}+4}
=(\tfrac34-o(1))n\cdot2^{M_n^{*}}$ (families counted twice lie in some
$\mathcal F_a^{*}\cap\mathcal F_b^{*}$, of exponent $3^{n-2}+4$). Hence
$R_{\mathrm{top}}=(\tfrac34+o(1))n\cdot2^{M_n^{*}}$.

\emph{Non-top part.} Every kernel-free family not contained in a single
$\mathcal F_a^{*}$ extends to a kernel-free MLS $M$ that is likewise not
any single $\mathcal F_a^{*}$, hence is non-extremal, hence has
$\Lambda(M)\le M_n^{*}-(2^{n-2}-4)$ by Theorem~\ref{thm:second}. Bounding
$R_{\mathrm{non\text-top}}$ by the union bound over all such MLS (the
total subfamily weight of a fixed MLS $M$ is $2^{\Lambda(M)}$),
\[
R_{\mathrm{non\text-top}}\;\le\;\lambda(n)\cdot2^{\,M_n^{*}-(2^{n-2}-4)}
\;=\;2^{M_n^{*}}\cdot\lambda(n)\,2^{-(2^{n-2}-4)}
\;=\;o\!\bigl(2^{M_n^{*}}\bigr),
\]
since $\log_2\lambda(n)=O(2^n/\sqrt n)=o(2^{n-2})$
(Lemma~\ref{lem:MLS-count}). Adding, $R(n)=R_{\mathrm{top}}
+R_{\mathrm{non\text-top}}=(\tfrac34+o(1))n\cdot2^{M_n^{*}}$.
\end{proof}

\begin{corollary}[Additive suppression exponent]\label{cor:additive}
As $n\to\infty$,
\[
\log_2\frac{Z_\cap(n)}{R(n)}
\;=\;2^{\,n-1}-2+\log_2\tfrac43+o(1).
\]
\end{corollary}

\begin{proof}
$Z_\cap(n)=n\cdot2^{3^{n-1}}(1+o(1))$ from~\eqref{eq:Zcap-formula} and
$R(n)=(\tfrac34+o(1))n\cdot2^{M_n^{*}}$ from
Theorem~\ref{thm:prefactor}; divide and take logarithms, using
$3^{n-1}-M_n^{*}=2^{n-1}-2$.
\end{proof}

\begin{corollary}[Sharpened Gibbs concentration]\label{cor:gibbs-sharp}
With $\mathbb P_n$ the Gibbs distribution of
Corollary~\ref{cor:gibbs}, the non-condensed probability has the exact
constant
\[
\mathbb P_n(\mathcal F\text{ non-condensed})
=\frac{R(n)}{W(n)}
=(3+o(1))\,2^{-2^{n-1}}.
\]
\end{corollary}

\begin{proof}
$R(n)/Z_\cap(n)=(\tfrac34+o(1))2^{-(2^{n-1}-2)}=(3+o(1))2^{-2^{n-1}}$ by
Theorem~\ref{thm:prefactor}, and $W(n)=Z_\cap(n)(1+o(1))$ since
$R(n)/Z_\cap(n)\to0$.
\end{proof}

This is an \emph{additive} $o(1)$: not merely the leading exponent but
the additive constant $\log_2\tfrac43$ in $\log_2(Z_\cap/R)$ is now
determined. The three quantitative facts behind it -- the extremal
value $M_n^{*}$, the single-sector constant $\tfrac34$, and the
second-level gap $2^{n-2}-4$ -- are all proved above. The quantity
$\varepsilon'_n:=\tfrac{4(\lambda(n)-n)}{3n}2^{-(2^{n-2}-4)}$, where
$\lambda(n)-n$ is the number of non-extremal kernel-free MLS, is the
\emph{upper-bound error term} from the crude union bound over
non-extremal MLS in Theorem~\ref{thm:prefactor}, not the true $o(1)$
correction; using the exact Hosten--Morris values of
$\lambda(n)$~\cite{BMMV2013,PawelskiSzepietowski2023} it is
$1.4\times10^{-1}$ at $n=6$, $1.0\times10^{-3}$ at $n=7$ and
$3.3\times10^{-8}$ at $n=8$, so the union bound becomes effective only
from $n\ge7$. The prefactor is correspondingly invisible at accessible
$n$: exhaustive enumeration gives
$R(5)=832\,931\,815\,524\,315\,163\,686$, that is
$R(5)/(\tfrac34\cdot5\cdot2^{67})\approx1.505$, consistent with
$\varepsilon'_5\approx1.18$. Determining the genuine $o(1)$ would require the
finer $\Lambda$-spectrum of the non-extremal MLS and the overlaps of
their subfamily contributions, not $\lambda(n)$ alone.

\subsection{
$B$-parameter
family}\label{sec:general-B}

A natural objection is that the constants above -- the extremal
exponent $M_n^{*}=3^{n-1}-2^{n-1}+2$, the gap $2^{n-2}-4$, and the
prefactor $\tfrac34$ -- might be artefacts of the specific weight
$w(S)=2^{2^{n-|S|}}-1$. They are not: the entire chain
$\text{C}\Rightarrow\text{A}\Rightarrow\text{B}$ goes through verbatim
for the whole family
\begin{equation}\label{eq:wB}
w_B(S)\;:=\;B^{\,B^{\,n-|S|}}-1,\qquad B\in\mathbb Z_{\ge2},
\end{equation}
of which the paper's weight is the case $B=2$. The point is that the
proofs never use the value $B=2$: they use only that
$1+w_B(S)=B^{\,B^{\,n-|S|}}$ is a power of a fixed base, so that the
subfamily generating function of a maximal linked system $M$ is
$\prod_{S\in M}(1+w_B(S))=B^{\Lambda_B(M)}$ with
\[
\Lambda_B(M):=\sum_{S\in M}B^{\,n-|S|},
\]
and that the per-pair coefficient $c_k^{(B)}=B^{\,n-k}-B^{k}$ is
\emph{positive} for every $k<n/2$ and every $B\ge2$ (since
$n-k>k$). The extremal and rigidity arguments of
\S\ref{sec:main-proof}, and the second-level argument of
\S\ref{sec:secondlevel}, depend only on these two facts.

\begin{theorem}[$B$-parameter version]\label{thm:generalB}
Fix an integer $B\ge2$ and let $Z_\cap^{(B)}(n)$, $R^{(B)}(n)$ be the
trivial and kernel-free parts of $Z(D_n,w_B)$. Then:
\begin{enumerate}[label=\textup{(\roman*)},leftmargin=*]
\item \emph{(Extremal exponent.)} The maximum of $\Lambda_B$ over
kernel-free maximal linked systems is
\[
M_n^{*}(B)\;=\;(B+1)^{n-1}-B^{n-1}+B,
\]
attained, for $n\ge5$, exactly by the $n$ one-flip stars $\mathcal
F_a^{*}$; for $n=4$ by those four together with the four triangle
systems of Theorem~\ref{thm:main-mu}(ii), all eight of profile
$(0,0,3,4,1)$ and hence of exponent $3B^{2}+4B+1=M_4^{*}(B)$; and for
$n=3$ the three $\mathcal F_a^{*}$ coincide. For $B=2$ the value is
$3^{n-1}-2^{n-1}+2$.
\item \emph{(Second-level gap.)} For $n\ge5$, every kernel-free maximal
linked system other than the $\mathcal F_a^{*}$ has
$\Lambda_B(M)\le M_n^{*}(B)-\bigl(B^{\,n-2}-B^{2}\bigr)$, with the same
equality set as in Theorem~\ref{thm:second} \emph{(the punctured-star systems there, and the balls $\mathcal A_T$
at $n=5$)}, independent of $B$;
for $B=2$ the gap is $2^{n-2}-4$.
\item \emph{(Exact prefactor.)} As $n\to\infty$,
\[
R^{(B)}(n)=\bigl(1-B^{-B}+o(1)\bigr)\,n\cdot B^{\,M_n^{*}(B)},
\qquad
\log_B\frac{Z_\cap^{(B)}(n)}{R^{(B)}(n)}
=B^{\,n-1}-B+\log_B\frac{B^{B}}{B^{B}-1}+o(1);
\]
for $B=2$ the constant is $1-2^{-2}=\tfrac34$ and
$\log_2\tfrac43$.
\end{enumerate}
\end{theorem}

\noindent\emph{Attribution.} Part~(i) is, via Corollary~\ref{cor:qary}
at $Q=B+1$, equivalent to the case $r=n$ of Borg's
theorem~\cite[Theorem~2.3]{Borg2013}, and is included with an
independent proof for completeness; parts~(ii) and~(iii) have no
counterpart in the product-model literature that we know of.

\begin{proof}
Throughout put $p:=1/(B+1)$ and $q=1-p=B/(B+1)$, so that
$0<p\le\tfrac13<\tfrac12$, and observe the specialisation identity
\begin{equation}\label{eq:specialise}
\Lambda_B(\mathcal F)
=\sum_{S\in\mathcal F}B^{\,n-|S|}
=(B+1)^{n}\sum_{S\in\mathcal F}p^{|S|}q^{\,n-|S|}
=(B+1)^{n}\,\mu_p(\mathcal F)
\end{equation}
for every family $\mathcal F$, valid because
$p^{|S|}q^{\,n-|S|}(B+1)^{n}=B^{\,n-|S|}$. Every statement proved about
$\mu_p$ at the single point $p=1/(B+1)$ therefore transfers to
$\Lambda_B$ upon multiplication by $(B+1)^{n}$; nothing below is argued
by analogy.

\emph{(i)} Theorem~\ref{thm:main-mu} at $p=1/(B+1)$ (equivalently,
Corollary~\ref{cor:qary} with $Q=B+1\ge3$) gives
$\max\mu_p=p-pq^{\,n-1}+qp^{\,n-1}$ over kernel-free MLS, attained
exactly by the $n$ one-flip stars, uniquely so for $n\ge5$ since
$p<\tfrac12$. Multiplying by $(B+1)^{n}$ term by term,
$(B+1)^{n}p=(B+1)^{n-1}$, $(B+1)^{n}pq^{\,n-1}=B^{\,n-1}$ and
$(B+1)^{n}qp^{\,n-1}=B$, whence
$M_n^{*}(B)=(B+1)^{n-1}-B^{\,n-1}+B$.

\emph{(ii)} Corollary~\ref{cor:second-mu} holds for \emph{every}
$0<p<\tfrac12$ and $n\ge5$, hence at $p=1/(B+1)$: over kernel-free MLS
other than the $\mathcal F_a^{*}$ the maximum of $\mu_p$ is
$p-pq^{\,n-2}+qp^{\,n-2}$, attained exactly by the punctured-star systems of
Theorem~\ref{thm:second} and, when $n=5$, additionally by the balls
$\mathcal A_T$. The equality set is the same for every
$p\in(0,\tfrac12)$, hence for every $B$ -- the asserted
$B$-independence. Multiplying by $(B+1)^{n}$, with
$(B+1)^{n}pq^{\,n-2}=(B+1)B^{\,n-2}$ and
$(B+1)^{n}qp^{\,n-2}=B(B+1)$, the second level of $\Lambda_B$ is
$(B+1)^{n-1}-(B+1)B^{\,n-2}+B(B+1)$, and
\[
M_n^{*}(B)-\Bigl[(B+1)^{n-1}-(B+1)B^{\,n-2}+B(B+1)\Bigr]
=(B+1)B^{\,n-2}-B^{\,n-1}-B^{2}
=B^{\,n-2}-B^{2}.
\]
For $B=2$ the gap is $2^{\,n-2}-4$.

\emph{(iii)} We re-run the proof of Theorem~\ref{thm:prefactor}; its
only weight-dependent inputs are three explicit quantities, computed
here for general $B$. \emph{Generating function:} for any MLS $M$,
$\prod_{S\in M}(1+w_B(S))=B^{\Lambda_B(M)}$, since
$1+w_B(S)=B^{\,B^{\,n-|S|}}$. \emph{Single sector:} a kernel-free
$\mathcal F\subseteq\mathcal F_a^{*}$ must contain
$A=[n]\setminus\{a\}$, and the total weight of subfamilies of
$\mathcal F_a^{*}$ containing $A$ is
\[
w_B(A)\!\!\prod_{S\in\mathcal F_a^{*}\setminus\{A\}}\!\!(1+w_B(S))
=\frac{w_B(A)}{1+w_B(A)}\,B^{\,M_n^{*}(B)}
=\bigl(1-B^{-B}\bigr)B^{\,M_n^{*}(B)},
\]
because $w_B(A)=B^{B}-1$. As in Proposition~\ref{prop:one-sector}, its
kernel-bearing part is at most
$(n-1)\,w_B(A)\,B^{(B+1)^{n-2}}$, since the members of
$\mathcal F_a^{*}\setminus\{A\}$ containing a fixed $j\ne a$ carry
$\Lambda_B$-total $\sum_{S\supseteq\{a,j\}}B^{\,n-|S|}=(B+1)^{n-2}$;
relative to the sector total this is
$O\bigl(n\,B^{-B[(B+1)^{n-2}-B^{\,n-2}]}\bigr)=o(1)$. Hence
$\Sigma_a^{(B)}=\bigl(1-B^{-B}\bigr)\bigl(1-o(1)\bigr)B^{\,M_n^{*}(B)}$.
\emph{Overlaps:} $\Lambda_B(\mathcal F_a^{*}\cap\mathcal F_b^{*})
=\sum_{S\supseteq\{a,b\}}B^{\,n-|S|}+2B=(B+1)^{n-2}+2B$ (the two
co-singletons contribute $B$ each), so inclusion--exclusion over the
$n$ sectors loses at most
$\binom n2\,B^{(B+1)^{n-2}+2B}=o\bigl(B^{M_n^{*}(B)}\bigr)$, giving
$R^{(B)}_{\mathrm{top}}=\bigl(1-B^{-B}+o(1)\bigr)\,n\,B^{\,M_n^{*}(B)}$.
For the non-top part, part~(ii) and the union bound over non-extremal
kernel-free MLS give
\[
R^{(B)}_{\mathrm{non\text-top}}
\;\le\;\lambda(n)\,B^{\,M_n^{*}(B)-(B^{\,n-2}-B^{2})}
\;=\;o\bigl(B^{\,M_n^{*}(B)}\bigr),
\]
since $\log_B\lambda(n)\le\log_2\lambda(n)=O(2^{n}/\sqrt n)
=o(2^{\,n-2})\le o(B^{\,n-2})$ (Lemma~\ref{lem:MLS-count}). Adding,
$R^{(B)}(n)=\bigl(1-B^{-B}+o(1)\bigr)\,n\,B^{\,M_n^{*}(B)}$. Finally,
the inclusion--exclusion of Proposition~\ref{prop:Zcap} applies
verbatim with $\prod_{S\ni i}(1+w_B(S))=B^{(B+1)^{n-1}}$, giving
\[
Z_\cap^{(B)}(n)=\sum_{j=1}^{n}(-1)^{j+1}\binom nj\,B^{(B+1)^{n-j}}
=n\,B^{(B+1)^{n-1}}(1+o(1));
\]
dividing and using $(B+1)^{n-1}-M_n^{*}(B)=B^{\,n-1}-B$ yields
\[
\log_B\frac{Z_\cap^{(B)}(n)}{R^{(B)}(n)}
=B^{\,n-1}-B+\log_B\frac{B^{B}}{B^{B}-1}+o(1).
\] For $B=2$ the constant
is $1-2^{-2}=\tfrac34$ and $\log_2\tfrac43$. All closed forms were
additionally verified against exhaustive enumeration for
$B\in\{2,3,4,5\}$ and $n\in\{5,6\}$ (Appendix~\ref{app:code}).
\end{proof}

Thus the phenomenon is structural, not arithmetic: for the whole family
\eqref{eq:wB}, the same $n$ stars dominate, the same near-extremal
families sit a fixed gap below, and the suppression prefactor is the
explicit $1-B^{-B}$. The paper's weight $B=2$ is the smallest and
sharpest case, but nothing in the argument privileges it.

\section{Concluding remarks}\label{sec:concl}

The doubly exponential weight~\eqref{eq:weight} concentrates $W(n)$
sharply: the $n$ stars attain the global maximum of $\mu_{p}$ among
intersecting families, the one-flip stars $\mathcal F_a^{*}$ are the
cheapest kernel-free configurations overall by
Theorem~\ref{thm:main-mu} (with the cost within the natural one-defect
family made explicit in Proposition~\ref{prop:defect}), and
Theorem~\ref{thm:main} shows that $W(n)$ is dominated by the $n$ stars
with doubly exponential suppression of all other configurations.

For the entry $r=2$, $p\le\tfrac12$, in the table of the quantities
$M_r(n,p)$, the closed form and the maximiser classification go back,
at the points $p=1/Q$, to Borg~\cite{Borg2013} (see
also~\cite{KwanSudakovVieira2018,FranklNie2026,HouHu2026}), with the
extremal family known earlier still as the Wheel coterie in the
availability theory of Peleg and Wool~\cite{PelegWool1995};
Theorem~\ref{thm:main-mu} records the statement in continuous $p$,
with an elementary proof. The other known entries of the table are due
to Brace--Daykin ($p=\tfrac12$)~\cite{BraceDaykin1971},
Frankl--Tokushige and Tokushige ($r\ge8$, $p$ near
$\tfrac12$)~\cite{FranklTokushige2006,Tokushige2008}, and Tokushige
($r=3$)~\cite{Tokushige2024}. Very recent work of
Wu and Feng approaches Tokushige's general $r$-wise conjecture via
random partitions~\cite{WuFeng2026rwise}. Ahlswede and Khachatrian
also determined the maximum \emph{non-trivial} $t$-intersecting
families exactly~\cite{AhlswedeKhachatrian1996}; porting that
classification to the $\mu_p$ setting -- extending
Theorem~\ref{thm:main-mu} from $t=1$ to all $t\ge2$, in the spirit of
Filmus's $\mu_p$ version of the complete intersection
theorem~\cite{Filmus2017} -- is in our view the most natural
continuation of the present work.



Natural directions for further work:

\begin{enumerate}[label=(\arabic*),leftmargin=*]
\item \emph{The full $\Lambda$-spectrum.} Theorem~\ref{thm:second}
determines the top two levels of the $\Lambda$-spectrum of kernel-free
MLS exactly (values $M_n^{*}$ and $M_n^{*}-(2^{n-2}-4)$, with no MLS
between them). Exhaustive enumeration gives the third level
$M_n^{*}-2(2^{n-2}-4)$ for $n=5,6,7$ (values $59$, $189$, $611$). For
$n\le6$ the entire top of the spectrum is arithmetic with step
$2^{n-2}-4$, namely $67,63,59,55,51$ and $213,201,189,177,165,153$;
this is a small-$n$ coincidence, since at $n=7$ the spectrum begins
$667,639,611,603,583,575$ and the step is not constant. Determining the
third level in general, and whether the local gaps are uniformly
bounded, is open.
\item \emph{The $o(1)$ rate in the prefactor.} Theorem~\ref{thm:prefactor}
and Corollary~\ref{cor:additive} determine the prefactor $\tfrac34$ and
the additive constant $\log_2\tfrac43$ exactly. The genuine $o(1)$
correction is not simply a function of $\lambda(n)$: it depends on the
whole $\Lambda$-spectrum of the non-extremal MLS and on the overlaps of
their subfamily contributions. Determining it -- even its order -- is
open, and would in particular require going beyond the second level
established in Theorem~\ref{thm:second}.
\item \emph{Higher $r$ and $t$.} The analogues of
Theorem~\ref{thm:main-mu} for non-trivial $r$-wise intersecting
families, $r\ge 4$, and for $t\ge 2$, remain open away from
$p=\tfrac12$; cf.~\cite{Tokushige2024} for $r=3$, the
counterexamples there to a conjecture of O'Neill and
Verstra\"ete~\cite{ONeillVerstraete2021}, and the $r$-cross
machinery of~\cite{GMPS}.
\item \emph{Stability.} For $\tfrac13\le p\le\tfrac12$ Tokushige proves
stability of the optimal structure for $r=3$; a quantitative stability
version of Theorem~\ref{thm:main-mu} (families of measure close to
$M_2(n,p)$ are close to a one-flip star) would sharpen the
enumerative application.
\end{enumerate}

\subsection*{Acknowledgements}
The author is grateful to A.~A.~Esin for discussions at an early
stage of this project, in particular in connection with the
three-valued-logic problem (Appendix~\ref{app:logic}) that motivated
the weight~\eqref{eq:weight}.

\subsection*{Declarations}
\noindent\textbf{Competing interests.} The author declares that she
has no competing interests.

\noindent\textbf{Declaration of generative AI and AI-assisted
technologies in the writing process.}
During the preparation of this work the author used a large language
model (Claude, Anthropic) to translate the manuscript from
Russian into English, to refine the English style and grammar, to
assist with \LaTeX{} formatting, and to draft
the auxiliary verification scripts reproduced in
Appendix~\ref{app:code}. After using this tool, the author reviewed,
edited and independently verified all content and takes full
responsibility for the content of the publication. The mathematical
results -- definitions, statements and proofs -- are due to the
author; all numerical data and bibliographic records were re-checked
by the author.

\smallskip\noindent\textbf{Funding.} This work was carried out within
the framework of the state assignment of the Institute for Information
Transmission Problems of the Russian Academy of Sciences (Kharkevich
Institute), project no.\ FFNU-2025-0029.

\smallskip\noindent\textbf{Data availability.} The numerical values in
Tables~\ref{tab:data}--\ref{tab:spectrum} and the extremal and spectral
enumerations for $n\le 6$ are reproducible from the script in
Appendix~\ref{app:code}.

\appendix

\section{A remark on the origin of the weight}\label{app:logic}

The specific weight $w(S)=2^{2^{n-|S|}}-1$ in~\eqref{eq:weight} is
\emph{motivated} by an enumeration problem in three-valued logic: in
the study of $\mathcal R_1$-closed subclasses of the layer
$T_{01}$~\cite{Kalimulina2022,Esin2008,Esin2024}, the quantity
$2^{n-|S|}$ counts the Boolean inputs lying over a ``fault-support''
$S\subseteq[n]$, and a pairwise-intersection condition on
fault-supports appears as a closure constraint. A precise statement
and proof of a correspondence between such subclasses and the
weighted count $W(n)$ is deferred to a separate paper; \textbf{we make
no such claim here, and nothing in the present paper depends on it}.
This remark is recorded only to indicate why this particular doubly
exponential weight, rather than an arbitrary one, seemed worth
studying.

\section{Exact numerical data}\label{sec:data}

Exact values of $W(n)$, $Z_\cap(n)$ and $R(n)$ are obtained as
follows. For $n\le 4$ all subfamilies of
$2^{[n]}\setminus\{\varnothing\}$ are enumerated exhaustively. The
total number of intersecting families is modest:
$|\Intn|=2,\,6,\,40,\,1376,\,1\,314\,816$ for $n=1,\dots,5$
(OEIS A051185); for $n=5$ this count, and the remainder $R(5)$, are
computed in seconds as weighted counts of independent sets of the
disjointness graph on the $31$ nonempty subsets of $[5]$ (for $R(5)$,
restricted to $L_{\ge 2}$ with the kernel-free condition handled by
inclusion--exclusion over the possible kernel element; any non-trivial
intersecting family avoids singletons by the argument of
Lemma~\ref{lem:singleton}). All values, as well as the extremal
enumerations behind Theorem~\ref{thm:main-mu} for $n\le 5$, are
reproducible by the script in Appendix~\ref{app:code}.

\begin{table}[ht]
\centering
\caption{Decomposition $W(n)=Z_\cap(n)+R(n)$ and the exponent gap.}
\label{tab:data}
\medskip
\begin{tabular}{@{}crrcc c@{}}
\toprule
$n$ & $Z_\cap(n)$ & $R(n)$ & $R/Z_\cap$
& $\log_2(Z_\cap/R)$ & $2^{n-1}\!-\!2$ (benchmark) \\
\midrule
$1$ & $2$ & $0$ & $0$ & -- & $-1$ \\
$2$ & $14$ & $0$ & $0$ & -- & $0$ \\
$3$ & $1\,514$ & $54$ & $0.036$ & $4.8$ & $2$ \\
$4$ & $536\,867\,870$ & $13\,158\,882$ & $0.025$ & $5.35$ & $6$ \\
$5$ & $\approx 1.21\!\times\! 10^{25}$
    & $\approx 8.33\!\times\! 10^{20}$ & $6.89\!\times\! 10^{-5}$
    & $13.83$ & $14$ \\
\bottomrule
\end{tabular}
\end{table}

Exact values for $n=5$:
\begin{align*}
Z_\cap(5)&=12\,089\,258\,196\,146\,290\,404\,889\,562,\\
R(5)&=832\,931\,815\,524\,315\,163\,686.
\end{align*}
The last column is the exponent benchmark of~\eqref{eq:rate} without
the multiplicity correction. The observed finite-$n$ gaps in the table are still
\emph{below} the asymptotic benchmark $2^{n-1}-2$ (e.g.\ $5.35<6$ at
$n=4$, $13.83<14$ at $n=5$): at small $n$ the system has not yet
reached the asymptotic additive regime, the prefactor and second-level
corrections not being negligible; asymptotically,
Corollary~\ref{cor:additive} gives $\log_2(Z_\cap/R)=2^{n-1}-2+
\log_2\tfrac43+o(1)$, and the second-level gap of
Theorem~\ref{thm:second} is exactly what makes the crude multiplicity
factor $\lambda(n)/(n-1)$ in~\eqref{eq:rate} asymptotically harmless:
most kernel-free MLS carry exponent at most $M_n^{*}-(2^{n-2}-4)$, not
$M_n^{*}$.

\begin{table}[ht]
\centering
\caption{The top of the $\Lambda$-spectrum of kernel-free MLS
(exponent value $^{\times}$ multiplicity), from exhaustive enumeration.
The maximum is $M_n^{*}$ and the second level is $M_n^{*}-(2^{n-2}-4)$,
matching Theorem~\ref{thm:second}; the two are separated by a gap with
no MLS inside it.}
\label{tab:spectrum}
\medskip
\begin{tabular}{@{}cll@{}}
\toprule
$n$ & $M_n^{*}$, second level & top $\Lambda$-values $^{\times}$ mult. \\
\midrule
$4$ & $21,\ (\text{no second level})$ & $21^{\times 8}$ \\
$5$ & $67,\ 63$ & $67^{\times5},\,63^{\times30},\,59^{\times30},\,55^{\times10}$ \\
$6$ & $213,\ 201$ & $213^{\times6},\,201^{\times30},\,189^{\times140},\,177^{\times480}$ \\
\bottomrule
\end{tabular}
\end{table}

\begin{example}[$n=3$]
The only kernel-free MLS on $[3]$ is
$M=\{\{1,2\},\{2,3\},\{1,3\},\{1,2,3\}\}$, the one-flip star
$\mathcal F_1^{*}$, with $\Lambda(M)=2+2+2+1=7=3^2-2^2+2$. The two
non-trivial intersecting families on $[3]$ are $M$ and
$M\setminus\{[3]\}$, each of weight $27$, giving $R(3)=54$.
\end{example}

\begin{example}[$n=4$: the exceptional extremal set]
The eight kernel-free MLS on $[4]$ all attain
$\Lambda=21=3^3-2^3+2$: the four one-flip stars and the four triangle
systems $\mathcal T_B$, $B\in\binom{[4]}{3}$. This is the case
$n=4$ of Theorem~\ref{thm:main-mu}(ii), verified independently by
exhaustive enumeration (Appendix~\ref{app:code}).
\end{example}

\section{Verification script}\label{app:code}

The following Python script verifies the extremal enumerations of
Theorem~\ref{thm:main-mu} for $n\le 5$ (in particular the eight
maximisers at $n=4$ and the uniqueness at $n=5$), the values
$|\mathrm{Int}_n|$ for $n\le 5$, and the full decomposition
$W(n)=Z_\cap(n)+R(n)$ of Table~\ref{tab:data}: exhaustively for
$n\le 4$, and $R(5)$ via a weighted independent-set recursion on
$L_{\ge 2}$ combined with inclusion--exclusion over the kernel
element.

{\small
\begin{verbatim}
from itertools import combinations, product
from fractions import Fraction

def all_MLS(n):
    full = frozenset(range(n)); pairs, seen = [], set()
    for r in range(1, n):
        for c in combinations(range(n), r):
            s = frozenset(c)
            if s not in seen and (full-s) not in seen:
                pairs.append((s, full-s)); seen |= {s, full-s}
    out = []
    for ch in product([0,1], repeat=len(pairs)):
        fam = [p[c] for p, c in zip(pairs, ch)]
        if all(a & b for i,a in enumerate(fam) for b in fam[i+1:]):
            out.append(frozenset(fam) | {full})
    return out

def ker(F):
    k = None
    for s in F: k = s if k is None else k & s
    return k

def mu(F, p, n):
    return sum(p**len(s)*(1-p)**(n-len(s)) for s in F)

for n in [3,4,5]:
    kf = [m for m in all_MLS(n) if not ker(m)]
    for p in [Fraction(1,5), Fraction(1,3), Fraction(2,5)]:
        bnd = p - p*(1-p)**(n-1) + (1-p)*p**(n-1)
        ext = [m for m in kf if mu(m,p,n) == bnd]
        assert all(mu(m,p,n) <= bnd for m in kf)
        print(n, p, "kernel-free:", len(kf), "extremal:", len(ext))
# expected: n=3 -> 1/1; n=4 -> 8 of 8; n=5 -> 5 of 76 (for p<1/2)

# |Int_n| as independent sets of the disjointness graph
def count_int(n):
    V = [m for m in range(1, 1 << n)]
    nbr = [sum(1 << j for j,b in enumerate(V) if b != a and a & b == 0)
           for a in V]
    memo = {}
    def cnt(av):
        if av == 0: return 1
        if av in memo: return memo[av]
        best = max((bin(nbr[v] & av).count("1"), v)
                   for v in range(len(V)) if av >> v & 1)
        d, v = best
        if d == 0: r = 1 << bin(av).count("1")
        else: r = cnt(av & ~(1 << v)) + cnt(av & ~(1 << v) & ~nbr[v])
        memo[av] = r; return r
    return cnt((1 << len(V)) - 1)

print([count_int(n) for n in range(1, 6)])
# expected: [2, 6, 40, 1376, 1314816]   (OEIS A051185)
\end{verbatim}
}

The second part recomputes Table~\ref{tab:data}.

{\small
\begin{verbatim}
from math import comb

def Zcap(n):
    return sum((-1)**(j+1)*comb(n,j)*2**(3**(n-j)) for j in range(1,n+1))

def W_brute(n):  # exhaustive, n <= 4
    sets = [frozenset(c) for r in range(1,n+1)
            for c in combinations(range(n),r)]
    tot = 0
    for mask in range(1 << len(sets)):
        fam = [sets[i] for i in range(len(sets)) if mask >> i & 1]
        if all(a & b for i,a in enumerate(fam) for b in fam[i+1:]):
            prod = 1
            for S in fam: prod *= 2**(2**(n-len(S))) - 1
            tot += prod
    return tot

def R_count(n):  # weighted independent sets on L>=2, minus kernel part
    V = [frozenset(c) for r in range(2,n+1)
         for c in combinations(range(n),r)]
    wt = [2**(2**(n-len(S)))-1 for S in V]
    nbr = [sum(1 << j for j, b in enumerate(V) if j != i and not a & b)
           for i, a in enumerate(V)]
    memo = {}
    def A(av):
        if av == 0: return 1
        if av in memo: return memo[av]
        d,v = max((bin(nbr[u]&av).count("1"),u)
                  for u in range(len(V)) if av>>u&1)
        if d == 0:
            r, u = 1, av
            while u:
                b = (u & -u).bit_length()-1; r *= 1+wt[b]; u &= u-1
        else:
            r = A(av & ~(1<<v)) + wt[v]*A(av & ~(1<<v) & ~nbr[v])
        memo[av] = r; return r
    K = sum((-1)**(j+1)*comb(n,j)
            * 2**(3**(n-j) - (2**(n-1) if j == 1 else 0))
            for j in range(1, n+1))
    return A((1<<len(V))-1) - K

for n in [3,4]:
    assert W_brute(n) == Zcap(n) + R_count(n)
    print(n, Zcap(n), R_count(n))
print(5, Zcap(5), R_count(5))
# expected: (3, 1514, 54); (4, 536867870, 13158882);
# (5, 12089258196146290404889562, 832931815524315163686)
\end{verbatim}
}

The third part enumerates all kernel-free maximal linked systems by a
depth-first choice over complementary pairs and computes the
$\Lambda$-spectra of Table~\ref{tab:spectrum}, confirming the second
level of Theorem~\ref{thm:second} and that no MLS lies in the gap.

{\small
\begin{verbatim}
from itertools import combinations
from collections import Counter

def enumerate_MLS(n):
    full = (1 << n) - 1
    reps = [S for S in range(1, full) if S & 1]        # one per complementary pair
    reps.sort(key=lambda S: min(bin(S).count("1"), n - bin(S).count("1")))
    out, chosen = [], []
    def dfs(i):
        if i == len(reps):
            out.append(tuple(chosen)); return
        for X in (reps[i], full ^ reps[i]):
            if all(X & Y for Y in chosen):
                chosen.append(X); dfs(i + 1); chosen.pop()
    dfs(0); return out, full

def Lam(M, n):  return 1 + sum(2**(n - bin(S).count("1")) for S in M)   # +1 for [n]
def ker(M, full):
    k = full
    for S in M: k &= S
    return k

for n in (4, 5, 6):
    sols, full = enumerate_MLS(n)
    kf = [M for M in sols if ker(M, full) == 0]
    spec = Counter(Lam(M, n) for M in kf)
    Mstar  = 3**(n-1) - 2**(n-1) + 2
    second = 3**(n-1) - 3*2**(n-2) + 6
    top = sorted(spec.items(), reverse=True)[:4]
    print(f"n={n}: M*={Mstar}, 2nd={second}, top spectrum={top}")
    assert top[0][0] == Mstar
    if n >= 5:
        assert sorted(spec, reverse=True)[1] == second   # nothing strictly between
# expected: n=5 -> (67:5, 63:30, 59:30, 55:10);
#           n=6 -> (213:6, 201:30, 189:140, 177:480)

# B-parameter generalisation (Theorem: extremal value, gap, prefactor constant)
def LamB(M, n, B):
    return 1 + sum(B**(n - bin(S).count("1")) for S in M)
for B in (2, 3, 4, 5):
    for n in (5, 6):
        sols, full = enumerate_MLS(n)
        kf = [M for M in sols if ker(M, full) == 0]
        stars = [M for M in kf if sum(1 for S in M if bin(S).count("1") == 2) == n - 1]
        star_val = LamB(stars[0], n, B)
        nonstar = [M for M in kf if sum(1 for S in M if bin(S).count("1") == 2) <= n - 2]
        second_val = max(LamB(M, n, B) for M in nonstar)
        assert star_val == (B + 1)**(n - 1) - B**(n - 1) + B     # extremal value
        assert star_val - second_val == B**(n - 2) - B**2        # gap
    # single-sector prefactor constant 1 - B^{-B}
    from fractions import Fraction as Fr
    assert Fr(B**B - 1, B**B) == 1 - Fr(1, B**B)
print("B-generalisation verified: M*(B)=(B+1)^(n-1)-B^(n-1)+B, gap=B^(n-2)-B^2,",
      "prefactor 1-B^-B, for B=2..5")

# Profile-level optimisation (Remark on the profile polytope): the top two DISTINCT
# Lambda-values over PROFILES of kernel-free MLS are exactly M* and M*-(2^{n-2}-4),
# and the second value is attained by a single profile (realised by two family types
# at n=5). This confirms the inequality is profile-level, the classification is not.
from collections import defaultdict
for n in (5, 6):
    sols, full = enumerate_MLS(n)
    kf = [M for M in sols if ker(M, full) == 0]
    prof_lam = defaultdict(set)
    for M in kf:
        pr = tuple(sum(1 for S in M if bin(S).count("1") == k) for k in range(1, n))
        prof_lam[pr].add(1 + sum(2**(n - bin(S).count("1")) for S in M))
    vals = sorted({next(iter(v)) for v in prof_lam.values()}, reverse=True)
    Mstar = 3**(n-1) - 2**(n-1) + 2
    assert vals[0] == Mstar and vals[1] == 3**(n-1) - 3*2**(n-2) + 6
    n_second = sum(1 for v in prof_lam.values() if next(iter(v)) == vals[1])
    print(f"n={n}: top profile-Lambda values {vals[:3]}, "
          f"second attained by {n_second} profile(s)")
\end{verbatim}
}

\section{A parity obstruction to naive layer-by-layer bounds}\label{sec:supplementary}

\subsection{A parity phenomenon for the layer-by-layer Hilton--Milner
bound}

It is natural to ask whether a layer-by-layer Hilton--Milner bound
could replace the MLS analysis. Define
\[
E_{\HM}(n)\;:=\;
\sum_{k=3}^{\lfloor n/2\rfloor}\HM(n,k)\cdot 2^{n-k}
+\sum_{k=\lfloor n/2\rfloor+1}^{n}\binom{n}{k}\cdot 2^{n-k},
\]
where $\HM(n,k)=\binom{n-1}{k-1}-\binom{n-k-1}{k-1}+1$ is the
Hilton--Milner bound~\cite{HiltonMilner1967}. Any intersecting
$\mathcal F\subseteq L_{\ge 3}$ that is non-trivial in each uniform
layer satisfies $\Lambda(\mathcal F)\le E_{\HM}(n)$; note that the
extremal families of Theorem~\ref{thm:main-mu} are \emph{not} of this
kind (their small layers are stars), so even a uniform bound
$E_{\HM}(n)<3^{n-1}$ would not recover~\eqref{eq:M-star}. However,
the inequality $E_{\HM}(n)<3^{n-1}$ itself fails for all large $n$:

\begin{remark}[Parity phenomenon]\label{rem:parity}
(i) \emph{Numerically}, the inequality $E_{\HM}(n)<3^{n-1}$ first fails
at $n=29$ among odd $n$
($E_{\HM}(29)-3^{28}\approx 2.45\cdot 10^{10}>0$) and at $n=46$ among
even $n$, and fails for all larger $n$ of either parity that we have
checked ($n\le 60$).
(ii) \emph{Provably}, it fails for all sufficiently large $n$.
Indeed, by Pascal's rule
$\binom{n}{k}-\binom{n-1}{k-1}=\binom{n-1}{k}$, one has the exact
decomposition
\begin{align*}
E_{\HM}(n)-3^{n-1}
&=\underbrace{\sum_{k>\lfloor n/2\rfloor}\binom{n-1}{k}2^{n-k}}_{\text{tail excess}}
-\underbrace{\sum_{3\le k\le\lfloor n/2\rfloor}
\Bigl(\binom{n-k-1}{k-1}-1\Bigr)2^{n-k}}_{\text{HM saving}}\\
&\qquad-\underbrace{\bigl(2^{n-1}+(n-1)2^{n-2}\bigr)}_{\text{absent layers }k=1,2}.
\end{align*}
The tail excess is $\Theta\bigl(2^{3n/2}n^{-1/2}\bigr)$ (the maximal
term sits at $k=\lceil (n+1)/2\rceil$). The saving equals $(1-o(1))\,t_n$, where
$t_n:=\sum_{k\ge1}\binom{n-k-1}{k-1}2^{n-k}$ has generating function
$\sum_n t_nx^n=\sum_{k\ge1}\bigl(2x^{2}/(1-2x)\bigr)^{k}$, a rational
function whose dominant singularity is the root
$x_0=(\sqrt3-1)/2$ of $2x^{2}+2x-1=0$; hence
$t_n=\Theta\bigl((1+\sqrt3)^n\bigr)
=2^{(\log_2(1+\sqrt3)+o(1))n}$ with
$\log_2(1+\sqrt3)\approx 1.4499<\tfrac32$; the absent layers
contribute only $O(n2^{n})$. Hence the positive tail term dominates
for all large $n$. The parity dependence arises because for
odd $n$ the layer $k=\lceil n/2\rceil$ falls outside the
Hilton--Milner range ($n<2k$) and contributes the full binomial
$\binom{n}{k}$. Theorem~\ref{thm:main-mu} bypasses this obstruction by
working directly at the MLS level.
\end{remark}

\subsection{A cross-layer constraint}

The next observation shows that layer-by-layer bounds are far from
tight once the structure of a single layer is known.

\begin{proposition}[Cross-layer constraint]\label{prop:cross-layer}
Let $\mathcal F\subseteq L_{\ge 3}$ be an intersecting family with
$\mathcal F\cap L_3\subseteq\Star(1)$ and
$|\mathcal F\cap L_3|=\binom{n-1}{2}$ (a maximum $3$-uniform star at
$1$). Then every $T\in\mathcal F$ with $1\notin T$ has $|T|\ge n-2$,
and consequently
\[
\Lambda(\mathcal F)\;\le\;
\bigl(3^{n-1}-2^{n-1}-(n-1)\cdot 2^{n-2}\bigr)
+\bigl((n-1)\cdot 4+2\bigr)\;<\;3^{n-1}\quad\text{for all }n\ge 4.
\]
\end{proposition}

\begin{proof}
If $1\notin T$, then for every triple $\{1,a,b\}\in\mathcal F\cap L_3$
we need $T\cap\{a,b\}\neq\varnothing$; the pairs $\{a,b\}$ range over
all of $\binom{[n]\setminus\{1\}}{2}$, and a set
$T\subseteq[n]\setminus\{1\}$ meeting all of them must miss at most one
element of $[n]\setminus\{1\}$, i.e.\ $|T|\ge n-2$. The exponent
estimate follows from
$\Lambda(\Star(1)\cap L_{\ge 3})=3^{n-1}-2^{n-1}-(n-1)\cdot 2^{n-2}$,
plus at most $n-1$ sets of size $n-2$ (each contributing $4$; this
term is vacuous for $n=4$) and one
set of size $n-1$ (contributing $2$) not containing~$1$.
\end{proof}


\section{Supplementary figures}\label{app:figs}

\begin{figure}[t]
\centering
\includegraphics[width=0.68\textwidth]{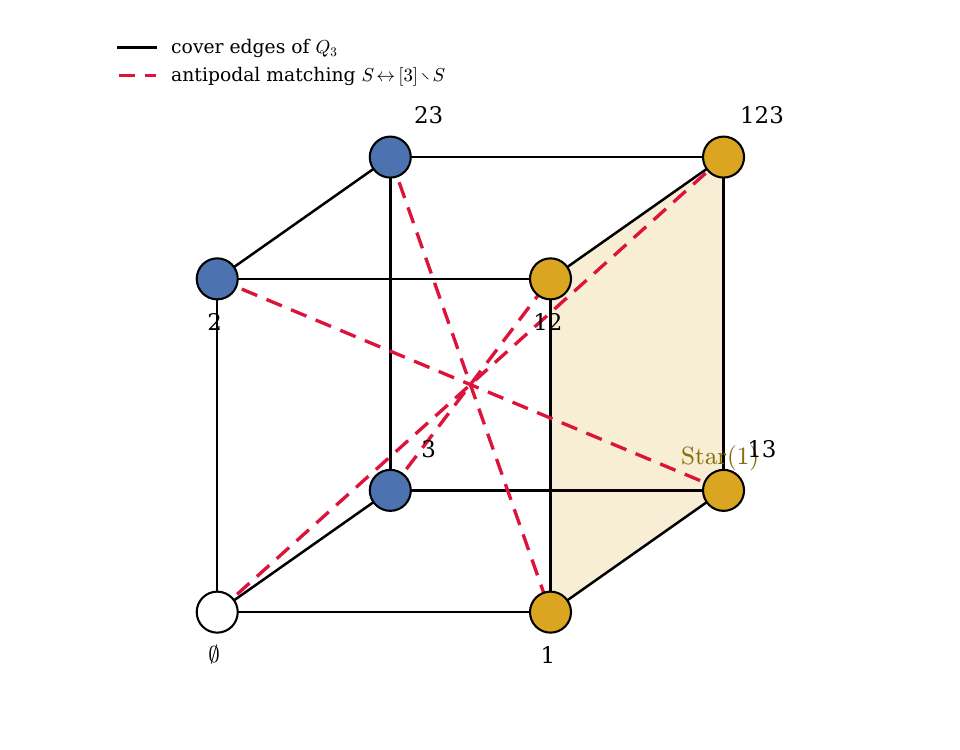}
\caption{The cube $Q_3=2^{[3]}$ with cover edges (black), the
complementary-pair matching $S\leftrightarrow[3]\setminus S$ (red), and
the facet $\Star(1)$ (shaded); the empty set $\varnothing$ (bottom
corner) is excluded from $V_3$ and carries no vertex of $D_3$.}
\label{fig:cube}
\end{figure}

\begin{figure}[t]
\centering
\includegraphics[width=0.72\textwidth]{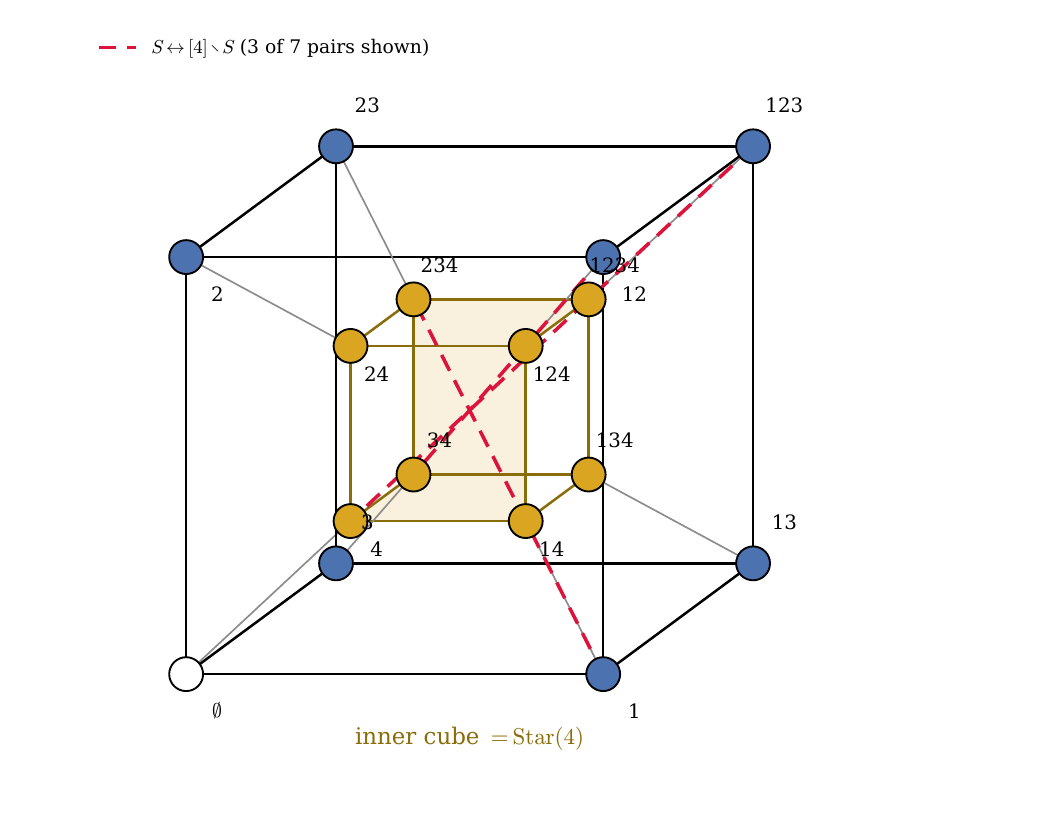}
\caption{The same structure one dimension up, on the tesseract
$Q_4=2^{[4]}$, drawn as the standard two-nested-cubes projection; the
empty set $\varnothing$ is again excluded from $V_4$. The inner cube is
exactly the facet $\Star(4)$; three complementary pairs are shown (of
the seven present) to keep the figure legible.}
\label{fig:tesseract}
\end{figure}

\end{document}